\documentclass[11pt]{amsart}
\usepackage[colorlinks, citecolor=blue, dvipdfm, pagebackref]{hyperref}

\usepackage{extarrows}

\newtheorem{theorem}{Theorem}[section]
\newtheorem{lemma}[theorem]{Lemma}

\theoremstyle{definition}
\newtheorem{definition}[theorem]{Definition}
\newtheorem{question}[theorem]{Question}
\newtheorem{example}[theorem]{Example}

\newtheorem{corollary}[theorem]{Corollary}
\newtheorem{remark}[theorem]{Remark}

\theoremstyle{remark}

\newcommand{\be}{\begin{equation}}
\newcommand{\ee}{\end{equation}}
\newcommand{\hooklongrightarrow}{\lhook\joinrel\longrightarrow}
\numberwithin{equation}{section}

\begin{document}

\title{Characterization on projective submanifolds of codimensions 2 and 3}

\author{Ping Li}

\address{School of Mathematical Sciences, Tongji University, Shanghai 200092, China}

\email{pingli@tongji.edu.cn\\
pinglimath@gmail.com}
\author{Fangyang Zheng}
\address{School of Mathematical Sciences, Chongqing Normal University, Chongqing
401331, China}
\email{franciszheng@yahoo.com}
\thanks{The first author was partially supported by the National
Natural Science Foundation of China (Grant No. 11722109).}

 \subjclass[2010]{57R20, 32Q55, 53C55, 57R22.}


\keywords{Chern class, Chern class inequality, very ample line bundle, codimension, second fundamental form,  Gauss map, tangent variety, secant variety.}

\begin{abstract}
In this article we give a necessary and sufficient condition to characterize projective submanifolds in ${\mathbb P}^N$ with codimensions $2$ and $3$. The conditions involve the Chern classes of the manifold and a very ample line bundle on the manifold. This generalizes our earlier characterization for hypersurfaces. The higher codimensional cases are proposed as a general question.
\end{abstract}

\maketitle

\tableofcontents

\section{Introduction and statements of results}\label{introduction}
In our recent article \cite{LZ}, we obtained the following characterization for hypersurfaces in complex projective space (\cite[Thm 2.2]{LZ}).
\begin{theorem}\label{thm1.1}
A complex projective manifold $X^n$ with dimension $n\geq 2$ can be realized as a hypersurface in ${\mathbb P}^{n+1}$ if and only if $X$ admits a very ample line bundle $L$ such that
\be\nonumber
\sigma_2(X,L):= \frac{1}{2}(n+2)(n+1)L^2 - (n+2)Lc_1+c_1^2-c_2=0
\ee
as a cohomology class, where $c_i$ is the $i$-th Chern class of $X$.
\end{theorem}
The purpose of this article is to generalize the above result and give characterization to projective submanifolds of codimensions two  and three.

Before stating the results, let us fix some notations. Let $X^n$ be a projective manifold of dimension $n$ and $L$ a very ample line bundle on $X$. Let $i: X\hookrightarrow {\mathbb P}^N $ be a holomorphic embedding. We will always assume that it is {\em non-degenerate}, namely, $i(X)$ is not contained in any hyperplanes of ${\mathbb P}^N $. The embedding $i$ is said to be {\em associated to} $L$, if $i^{\ast} \big( {\mathcal O}_{{\mathbb P}^N}(1)\big) = L$.

If we take a basis $\{ s_0, \ldots , s_{N_0}\}$ of $H^0(X,L)\cong {\mathbb C}^{N_0+1}$, the space of all global holomorphic sections of $L$ on $X$, then we get a non-degenerate holomorphic embedding $i_0: X\hookrightarrow {\mathbb P}^{N_0} $ via
\be\nonumber\begin{split}
X&\overset{i_0}{\hooklongrightarrow}
\mathbb{P}\big(H^0(X,L)^{\ast}\big)\cong\mathbb{P}^{N_0},\\
x&\longmapsto[s_0(x):s_1(x):\cdots:s_{N_0}(x)].
\end{split}
\ee
This $i_0$ will be called the {\em Kodaira map} of $L$.

Clearly $N_0\geq N$, and any non-degenerate embedding $i$ associated to $L$ is given by $i=\pi \circ i_0$ where $\pi : {\mathbb P}^{N_0} \setminus P_1 \rightarrow P_2 \cong {\mathbb P}^N$ is the projection determined by a linear subspace $P_1 \cong {\mathbb P}^{N_0-N-1}$ ($\mathbb{P}^{-1}:=\{pt\}$) in ${\mathbb P}^{N_0}$ which does not intersect $i_0(X)$. More precisely,  $P_1$ does not intersect $\text{Sec}\big(i_0(X)\big)$, the secant variety of $i_0(X)$, and more details can be found in the proof of Lemma \ref{projection}.

Denote by
$$r_L:=\min\big\{N-n~|~\text{$X\hookrightarrow {\mathbb P}^N$ associated to $L$}\big\},\qquad r_X:=\min\big\{r_L~|~\text{very ample $L$}\big\},$$
and call them the {\em codimensions} of $(X,L)$ and $X$, respectively.

Once we have  an embedding $i$ associated to $L$, it will induce a \emph{Gauss map} $\gamma$ sending $x\in X$ to the $n$-dimensional \emph{projective} tangent space $\widetilde{T}_xX$ in $\mathbb{P}^N$, which is defined as the limiting position of all chords $\overline{xy}$ as $y\rightarrow x$, and can be uniquely identified with an $(n+1)$-dimensional linear subspace in $\mathbb{C}^{N+1}$:
\be\label{Gauss map}
\begin{split}
X&\overset{\gamma}{\longrightarrow}
\mathbb{G}_{n}(\mathbb{P}^N)\cong G_{n+1}(\mathbb{C}^{N+1})\\
x&\longmapsto\widetilde{T}_xM,
\end{split}\ee
where $\mathbb{G}_{n}(\mathbb{P}^N)$ is the Grassmannian variety of $n$-dimensional projective subspaces in $\mathbb{P}^N$, which is isomorphic to $G_{n+1}(\mathbb{C}^{N+1})$, the usual complex Grassmannian of $(n+1)$-dimensional linear spaces in $\mathbb{C}^{N+1}$. Denote by $Q$ the universal quotient bundle of $G_{n+1}(\mathbb{C}^{N+1})$. It turns out that the Chern class of $\gamma^{\ast}Q$, the pull back of $Q$ by $\gamma$, depends only on the pair $(X,L)$ and is independent on the choice of the embedding $i$ associated to $L$. Indeed, as a straightforward computation, we get in \cite[\S 5]{LZ} that the $k$-th Chern class of $\gamma^{\ast }Q$ is given by
\be
\sigma_k(X,L):= c_k(\gamma^{\ast }Q) = \sum_{i=0}^k\binom{n+k}{k-i}\cdot L^{k-i}\cdot s_i(X),\nonumber
\ee
where $s_i(X)$ is the $i$-th Segre class of the holomorphic tangent bundle $TX$, defined as the $i$-th component of the  formal inverse $s(X)$ of the total Chern class $c(X)$. Namely, $s(E)c(E)=1$ for any vector bundle $E$. In particular,
\begin{eqnarray}\nonumber
\sigma_1(X,L)&=&(n+1)L-c_1;\\
\nonumber
\sigma_2(X,L)&=&\frac{1}{2}(n+2)(n+1)L^2 - (n+2)Lc_1+c_1^2-c_2;\\
\nonumber
\sigma_3(X,L)& = & \binom{n+3}{3} L^3 - \binom{n+3}{2} L^2c_1 + (n+3) L(c^2_1-c_2) - (c_1^3-2c_1c_2+c_3); \\
\nonumber
\sigma_4(X,L) & = & \binom{n+4}{4} L^4 - \binom{n+4}{3} L^3c_1 + \binom{n+4}{2} L^2 (c^2_1-c_2) + \\
\nonumber
& & -\ (n+4) \,L(c_1^3-2c_1c_2+c_3) + (c_1^4 -2c_1^2c_2+ c_1c_3-c_2^2+c_4).
\end{eqnarray}

Since $Q$ is globally generated and so is $\gamma^{\ast }Q$, it is well-known that $\sigma_k(X,L) \geq 0$
as a cohomology class. Namely, $\int_Y \sigma_k(X,L)\geq 0$ for any $k$-dimensional irreducible subvariety $Y\subset X$. Moreover, it turns out that the $k$-th Chern form of $\gamma^{\ast}(Q)$ with respect to the canonical connection is a nonnegative $(k,k)$-form in the strong sense (\cite[Prop. 3.1]{Li20}).

In the mean time, the rank of $\gamma^{\ast}(Q)$ is $N-n$, the codimension of $X^n$ in ${\mathbb P}^N$. So we know that $\sigma_k(X,L)=0$ whenever $k>\min\{n, N-n\}$. In particular, for any integer $n\geq k\geq 2$, if $r_L<k$, then $\sigma_k(X,L)=0$. It is natural to ask if the converse is true:
\begin{question}\label{question1}
Let $k\geq 2$ be an integer, and $L$ a very ample line bundle over a projective manifold $X^n$ with $n\geq k$ such that $\sigma_{k}(X,L)=0$. Is it true that $r_L<k$?
\end{question}

A possibly weaker question would be, under the same assumption, is it true that $r_X<k$? In \cite[Thm 2.2]{LZ} ($=$ Theorem \ref{thm1.1}), we have already seen the affirmative answer to Question \ref{question1} for $k=2$. Namely, the equality $\sigma_2(X,L)=0$ implies that $r_L<2$ so $X$ can be embedded in ${\mathbb P}^{n+1}$.

The first main result in this article is that the same is true for $k=3$, namely we have
\begin{theorem} \label{thm1.3}
If $X^n$ is a projective manifold with $n\geq 3$ and $L$ a very ample line bundle on it with $\sigma_3(X,L)=0$, then $r_L<3$. Consequently, a projective manifold $X^n$ with $n\geq3$ can be embedded in ${\mathbb P}^{n+2}$ if and only if it admits a very ample line bundle $L$ such that $\sigma_3(X,L)=0$.
\end{theorem}

This gives a necessary and sufficient condition characterizing submanifolds in projective space of codimension less than or equal to two.

As we shall see in the proofs (see Corollary \ref{codimbound}), $r_L$ is always no larger than the rank of the second fundamental form of $X\subset {\mathbb P}^N$, which we denote by $l$, when $l$ is less than the dimension of $X$. The reason that we have the above result for $k=2$ and $k=3$ cases is simply because $l<k$ under the assumption $\sigma_k(X,L)=0$. When $k\geq 4$, $l$ is no longer always less than $k$. For instance, the Segre fourfold $X={\mathbb P}^2\times {\mathbb P}^2 \subset {\mathbb P}^8$ satisfies $\sigma_4(X,L)=0$ but $l=4$. In this case, we still have $r_L=3<4$, since the secant variety of $X$ is $7$-dimensional. So when $k\geq 4$ one needs to dig deeper into the structure of the manifolds. As a trial case analysis, we obtain the following affirmative answer to the $k=4$ case of Question \ref{question1} when $n\geq 5$:

\begin{theorem} \label{thm1.4}
 If $X^n$ is a projective manifold of dimension $n\geq 5$ and $L$ a very ample line bundle on $X$ with $\sigma_4(X,L)=0$, then $r_L<4$. Consequently, a projective manifold $X^n$ with $n\geq 5$ can be embedded in ${\mathbb P}^{n+3}$ if and only if it admits a very ample line bundle $L$ with $\sigma_4(X,L)=0$.
\end{theorem}

We believe that Theorem \ref{thm1.4} should still hold when $n=4$, but our proof only covers the $n\geq 5$ case.

\subsection*{Organization of this article}
The rest of this article is organized as follows. In Section \ref{reduction} we reduce the restriction $\sigma_k(X,L)=0$ to an algebraic condition (Lemma \ref{reduce algebraic lemma}) as well as relate it to the second fundamental form of $X$ in
$\mathbb{P}^N$ ( Lemma \ref{second fundamental form lemma}). Then in Section \ref{algebraic question} a related algebraic question is proposed, and the solutions in the cases of $k=3$ and $k=4$, Lemmas \ref{codim2} and \ref{codim3}, are the main technical results in this article. By applying them the proofs of Theorems \ref{thm1.3} and \ref{thm1.4} are presented respectively in Sections \ref{proof 1} and \ref{proof 2}. The reasoning process of Lemmas \ref{codim2} and \ref{codim3} is routine but quite tedious, so for the conciseness of this article we put them to the last three appendices.

\section{Algebraic reduction}\label{reduction}
The main purpose of this section is to reduce the condition $\sigma_k(X,L)=0$ to an algebraic result, and relate this algebraic result to the second fundamental form of $X$ in $\mathbb{P}^N$.

First note that $\gamma^{\ast}(Q)$ is a quotient of the trivial bundle $\underline{\mathbb{C}}^{N+1}$, and we endow it with the induced metric from the trivial one on $\underline{\mathbb{C}}^{N+1}$. Then the $k$-th canonical Chern \emph{form} of this Hermitian vector bundle $\gamma^{\ast}(Q)$, denoted by $C_k(\gamma^{\ast}(Q))$, is a nonnegative $(k,k)$-form in the strong sense (cf. \cite[Prop. 3.1]{Li20}). So the condition $\sigma_k(X,L)=0$ is equivalent to $C_k(\gamma^{\ast}(Q))=0$ \emph{pointwisely} as a form, from which we shall derive the restrictions.
To this end, we shall calculate the curvature matrix of $\gamma^{\ast}(Q)$ under some local frame and \emph{explicitly} write down $C_k(\gamma^{\ast}(Q))$.

As before, let $L$ be a very ample line bundle over $X^n$ and $X\subset {\mathbb P}^N$ be a non-degenerate embedding associated to $L$. Let $[Z_0:\cdots :Z_N]$ be a homogeneous coordinate on $\mathbb{P}^N$. Denote by $\gamma : X \rightarrow G_{n+1}(\mathbb{C}^{N+1})$ the Gauss map, and by $S$ and $Q$ the universal subbundle and quotient bundle of $G_{n+1}(\mathbb{C}^{N+1})$, and we have the bundle exact sequence on $X$:
\be\label{exact sequence}0 \rightarrow \gamma^{\ast}S \rightarrow \underline{\mathbb{C}}^{N+1} \rightarrow \gamma^{\ast}Q \rightarrow 0 .\ee

Let us for later convenience fix the index range throughout the article:
\be\label{index range}
1\leq i,j,k\leq n, \qquad n+1\leq \alpha, \beta \leq N,\qquad 1\leq a,b\leq N,
\ee
and sometimes make the Einstein summation convenience if no confusions arise.

For a fixed point $x\in X$, we can take a unitary change of $Z$, which for the sake of simplicity we will still denote by $Z$, so that $x=[1:0:\cdots :0]$ and near $x$ the manifold $X$ is defined by $z^{\alpha}=f^{\alpha}(z^1,\ldots,z^n)$ for each $n+1\leq \alpha \leq N$, where $z^a:=\frac{Z_a}{Z_0}$ for $1\leq a\leq N$,  $(z^1, \ldots , z^n)$ gives a local holomorphic coordinate in $X$ centered at $x$, and the holomorphic functions $f^{\alpha}$ satisfy
\be\label{f convention}f^{\alpha}(0)= 0, \ \ f^{\alpha}_i (0):= \frac{\partial f^{\alpha}} {\partial z^i} (0)=0,  \ \ \ \ \forall \ n+1\leq \alpha \leq N, \ \ \forall  \  1\leq i\leq n.\ee

\subsection{The reduction of the condition $\sigma_k(X,L)=0$}\label{subsection2.1}
For technical reasons we treat the dual bundle $\gamma^{\ast}Q^{\ast}$ instead of $\gamma^{\ast}(Q)$. Our first lemma is the following
\begin{lemma}\label{curvature lemma}
There exists a local holomorphic frame of $\gamma^{\ast}Q^{\ast}$ around $x$ such that the curvature matrix $\Theta=(\Theta_{\alpha\beta})$ ($n+1\leq\alpha,\beta\leq N$) at $x$ is given by
\be\label{curvature matrix}
\Theta_{\alpha \beta}(x)= -\sum_{j=1}^n \xi^{\alpha}_j\wedge\overline{\xi^{\beta}_j},\qquad \xi^{\alpha}_j:=\partial f^{\alpha}_j(0) = \sum_{i=1}^n f^{\alpha}_{ij}(0)dz^i,\qquad f^{\alpha}_{ij}:=\frac{\partial^2 f^{\alpha}} {\partial z^i \partial z^j}.
\ee
\end{lemma}

\begin{proof}
The proof is similar to the arguments in \cite[\S 7.2]{LZ}, where a curvature matrix for $\gamma^{\ast}(S)$ is presented. For the reader's convenience, we still include a proof here.

Let $\{ e_0, e_1, \ldots , e_N\}$ be the standard frame of the trivial bundle $\underline{\mathbb{C}}^{N+1}$ on $X$ in (\ref{exact sequence}), and $\{ e^{\ast}_0, e^{\ast}_1, \ldots , e^{\ast}_N\}$ the frame in the trivial bundle $(\underline{\mathbb{C}}^{N+1})^{\ast}$ dual to $\{ e_0, \ldots , e_N\}$. In a neighborhood $x\in U\subset X$, a local holomorphic frame of $\gamma^{\ast}S$ is given by $\{ V_0, V_1, \ldots , V_n\}$ where
$$
V_0=e_0 + z^ie_i + f^{\alpha} e_{\alpha};\qquad
V_j =  e_0 + e_j + f^{\alpha}_j e_{\alpha}, \ \ 1\leq j\leq n.$$

Write
$$w=1-\sum_{i=1}^nz^i,\ \ \ U_0=\frac{1}{w} (V_0-\sum_{i=1}^nz^iV_i),\ \ \ U_j=V_j-U_0\ \ (1\leq j\leq n),$$
 we get a new local holomorphic frame $\{ U_0, U_1, \ldots , U_n\}$ of $\gamma^{\ast}S$ such that
$$U_j = e_j + h^{\alpha}_j e_{\alpha}, \ \ \ 0 \leq j \leq n, $$
where
$$
h^{\alpha}_0 = \frac{1}{w} (f^{\alpha}- z^if^{\alpha}_i);\qquad
h^{\alpha}_j = f^{\alpha}_j - h^{\alpha}_0 , \ \ 1\leq j\leq n.
$$

Let
$$Y_{\alpha} = e^{\ast }_{\alpha} - \sum_{j=0}^n h^{\alpha}_j e^{\ast }_j,\ \ n+1\leq \alpha \leq N.$$
Then $\{ Y_{n+1 }, \ldots , Y_N\}$ forms a local holomorphic frame of $\gamma^{\ast}Q^{\ast}$. Let $g=\langle \, , \rangle $ be the restriction on $\gamma^{\ast}Q^{\ast}$ of the flat metric of the trivial bundle $(\underline{\mathbb{C}}^{N+1})^{\ast}$ so that $\{ e^{\ast}_0, e^{\ast}_1, \ldots , e^{\ast}_N\}$ is unitary, then
\be\label{g}g=(\langle Y_{\alpha}, \overline{Y_{\beta}} \rangle) = (\delta_{\alpha \beta} + \sum_{j=0}^n h^{\alpha}_j \overline{h^{\beta}_j})=I_{N-n}+F\overline{F^t},\ \ F=(h^{\alpha}_j).\ee
Here and in what follows $\delta_{\alpha\beta}$ always denotes the Kronecker delta, $I_r$ the identity matrix of rank $r$, and ``$t$" the transpose of a matrix.

At the origin $x$, we have
\begin{eqnarray}\label{value at zero}
\left\{\begin{array}{ll}
h^{\alpha}_0(0)=0,\ \ dh^{\alpha}_0(0)=0,\qquad(n+1\leq\alpha\leq N)\\
~\\
h^{\alpha}_j(0)=0,\ \ dh^{\alpha}_j(0)=df^{\alpha}_j(0), \qquad (n+1\leq\alpha\leq N, \ 1\leq j\leq n)
\end{array} \right.
\end{eqnarray}
and so
\be\label{g0}g(0)=0,\ \ \ dg(0)=(0).\ee
Thus at the origin $x$ the curvature matrix of $\gamma^{\ast}Q^{\ast}$ under the frame $\{ Y_{n+1 }, \ldots , Y_N\}$ is given by
\be\begin{split}
\big(\Theta_{\alpha\beta}\big)(0)&=\bar{\partial}\big[(\partial g)\cdot g^{-1}\big](0)\\
&=- \partial F \wedge (\overline{\partial F})^t(0)\qquad\big(\text{by (\ref{g}) and (\ref{g0})}\big)\\
&=-\big(\partial h^{\alpha}_i\big)\wedge\big(\overline{\partial h^{\beta}_{j}}\big)^t(0)\\
&=\Big(-\sum_{j=0}^{n}\partial h^{\alpha}_j\wedge
\overline{\partial h^{\beta}_j}\Big)(0)\\
&=\Big(-\sum_{j=1}^{n}\partial f^{\alpha}_j\wedge
\overline{\partial f^{\beta}_j}\Big)(0),\qquad\big(\text{by (\ref{value at zero})}\big)
\end{split}\nonumber\ee
which gives the desired (\ref{curvature matrix}) and thus completes the proof.
\end{proof}

\begin{definition}\label{def}
Denote the symmetric $n\times n$ matrix $H^{\alpha}:=\big(f^{\alpha}_{ij}(0)\big)$ $(n+1\leq\alpha\leq N)$. For any column vector $u$, let us write $H^{\alpha}_u:=H^{\alpha}\cdot u\in {\mathbb C}^n$. For $k$ column vectors $u^{(1)},\ldots,u^{(k)}$, we simply denote the linear dependence of them by $u^{(1)}\wedge\cdots\wedge u^{(k)}=0$. As usual we denote by
$\ker(H^{\alpha}):=\{u\in\mathbb{C}^n~|~H^{\alpha}_u=0\}$.
\end{definition}

Since the condition $\sigma_k(X,L)=0$ is trivial when $k>N-n$. So we focus on the cases $k\leq N-n$. With Lemma \ref{curvature lemma} in hand, we are now ready to show the following
\begin{lemma}\label{reduce algebraic lemma}
With the above notation and symbols understood and assume that $2\leq k\leq N-n$. Then $\sigma_k(X,L)=0$ is equivalent to
\be
H^{\alpha_1}_u \wedge \cdots \wedge H^{\alpha_k}_u = 0, \ \ \ \ \  \forall \ n+1\leq \alpha_1 < \cdots < \alpha_k\leq N, \ \ \forall \  u \in {\mathbb C}^n. \label{eq:width2}
\ee
\end{lemma}

\begin{proof}
Denote by $S_k$ the group of symmetry for $k$ elements, and write $\alpha_I=(\alpha_1, \ldots ,\alpha_k)$ where $n+1\leq \alpha_1 < \cdots < \alpha_k\leq N$ for any multi-index of length $k$. Then  at the origin $x$, The Chern-Weil theory via Lemma \ref{curvature lemma} tells us that
\begin{eqnarray*}
&  &(\frac{2\pi}{\sqrt{-1}})^kC_k(\gamma^{\ast}Q^{\ast})\\
&= &  \sum_{ \alpha_I}\sum_{\tau \in S_k} \mbox{sgn} (\tau) \,\Theta_{\alpha_1 \alpha_{\tau (1)}} \wedge \cdots \wedge \Theta_{\alpha_k \alpha_{\tau (k)}} \\
& = &  \sum_{ \alpha_I}\sum_{\tau \!,\sigma \in S_k} \frac{1}{k!} \mbox{sgn} (\tau) \mbox{sgn} (\sigma) \,\Theta_{\alpha_{\sigma (1)} \alpha_{\tau (1)}} \wedge \cdots \wedge \Theta_{\alpha_{\sigma (k)} \alpha_{\tau (k)}} \\
& = & \sum_{ \alpha_I} \sum_{\tau\!,\sigma \in S_k} \sum_{j_1\!,\ldots , j_k=1}^n  \frac{(-1)^k}{k!} \mbox{sgn} (\tau ) \mbox{sgn} (\sigma)  \xi^{\alpha_{\sigma (1)}}_{j_1}  \wedge\overline{ \xi_{j_1}^{\alpha_{\tau (1)}} } \wedge \cdots \wedge \xi^{\alpha_{\sigma (k)}}_{j_k} \wedge \overline{ \xi_{j_k}^{\alpha_{\tau (k)}} } \\
& = &  \sum_{\alpha_I,\,\tau ,\sigma ,\, j_i} \frac{(-1)^k}{k!} \mbox{sgn} (\tau ) \mbox{sgn} (\sigma)  (-1)^{\frac{1}{2}k(k-1)}\xi^{\alpha_{\sigma (1)}}_{j_1} \wedge\cdots\wedge \xi^{\alpha_{\sigma (k)}}_{j_k}\wedge \overline{ \xi_{j_1}^{\alpha_{\tau (1)}} }\wedge \cdots \wedge \overline{ \xi_{j_k}^{\alpha_{\tau (k)}} } \\
& = & (-1)^k k! \sum_{\alpha_I,J} (-1)^{\frac{1}{2}k(k-1)} \Psi^{\alpha_I}_J \wedge \overline{\Psi^{\alpha_I}_J},
\end{eqnarray*}
where $J=(j_1, \ldots , j_k)$ with each $j_i$ running from $1$ to $n$, and for fixed $\alpha_I$ and $J$,
$$ \Psi^{\alpha_I}_J = \frac{1}{k!} \sum_{\pi \in S_k} \xi^{\alpha_1}_{j_{\pi (1)}} \wedge \cdots \wedge \xi^{\alpha_k}_{j_{\pi (k)}}.$$

Since as a form
$$\,(\sqrt{-1})^k (-1)^{\frac{1}{2}k(k-1)} \Psi^{\alpha_I}_{J} \wedge \overline{\Psi^{\alpha_I}_{J}} \geq 0$$
for each $\alpha_I$ and $J$, we know that the $(k,k)$-form
$ (-1)^k C_k( \gamma^{\ast}Q^{\ast} ) \geq 0$
everywhere on $X$, and when $\sigma_k(X,L) = c_k(\gamma^{\ast}Q)=0$, this form is identically zero as $c_k(\gamma^{\ast}Q)=0$ is represented by it, so each $\Psi^{\alpha_I}_J=0$. This means that for any given $n+1\leq \alpha_1< \cdots <\alpha_k\leq N$ and any $1\leq j_i\leq n$, $i=1, 2, \ldots , k$, we have
\be
\sum_{\pi \in S_k} \xi^{\alpha_1}_{j_{\pi (1)}} \wedge \cdots \wedge \xi^{\alpha_k}_{j_{\pi (k)}} = 0. \label{eq:width}
\ee
At the origin $x$, we have $\xi^{\alpha}_j = \sum_{i=1}^n f^{\alpha}_{ij}(0) dz_i$ and $f^{\alpha}_{ij}(0)$ are the components of the matrices $H^{\alpha}$. Now it is not hard to see that the condition (\ref{eq:width}) is equivalent to (\ref{eq:width2}).
\end{proof}

The condition (\ref{eq:width2}) does not mean much if the matrices $H^{\alpha}$ are highly degenerate. But in our case they are not as we have the following fact, which gives another restriction on the matrices $H^{\alpha}$ for \emph{generic} $x$.
\begin{lemma}\label{generic intersection}
For generic $x\in X$, we have
\be
\bigcap_{\alpha =n+1}^N \ker (H^{\alpha}) = 0.  \label{eq:nondeg}
\ee
\end{lemma}

\begin{proof}
First note that $(X,L)\neq({\mathbb P}^n, {\mathcal O}_{\mathbb{P}^n}(1))$, otherwise $N=n$, a contradiction.
Recall a classical fact that (\cite[p. 159]{BS}), for an ample line bundle $L$, the line bundle $(n+1)L+K_X$ is always ample, provided that $(X,L)\neq({\mathbb P}^n, {\mathcal O}_{\mathbb{P}^n}(1))$.  Here $K_X$ is the canonical line bundle of $X$. This implies that $$\sigma_1(X,L)=c_1(\gamma^{\ast}Q)=(n+1)L-c_1$$
is ample.

By Lemma \ref{curvature lemma} we know that $c_1(\gamma^{\ast}Q)$ is represented by the following nonnegative $(1,1)$-form:
$$ - C_1(\gamma^{\ast}Q^{\ast})= - \frac{\sqrt{-1}}{2\pi} \sum_{\alpha =n+1}^N \Theta_{\alpha \alpha} =  \frac{\sqrt{-1}}{2\pi} \sum_{\alpha =n+1}^N \sum_{j=1}^n \xi^{\alpha}_j \wedge \overline{\xi^{\alpha}_j}. $$
So at any generic point $x$ in $X$, this $(1,1)$-form must be positive-definite. That is, for any $(1,0)$-type tangent vector $0\neq\sum_{i=1}^nu_i\frac{\partial}{\partial z^i}$ at $x$, there exists some $\alpha$ and some $j$ such that
$$0\neq\xi^{\alpha}_j(\sum_{i=1}^nu_i\frac{\partial}{\partial z^i})=\sum_{i=1}^nf^{\alpha}_{ij}(0)u_i,$$
which, in the notation of Definition \ref{def}, is equivalent to the fact that for any column vector $0\neq u\in\mathbb{C}^n$ we have $H^{\alpha}_u\neq0$ for some $\alpha$. This is exactly (\ref{eq:nondeg}).
\end{proof}

\subsection{The second fundamental form of $X$}
In this subsection we relate the matrices $H^{\alpha}=\big(f^{\alpha}_{ij}(0)\big)$ to the second fundamental form of $X$ in $\mathbb{P}^N$. More precisely, we show in Lemma \ref{second fundamental form lemma} that they can be realized as the coefficients of this second fundamental form under some local frame field, which shall be used in the proof of the Theorems stated in  Section \ref{introduction}.

Now let ${\mathbb P}^N$ be endowed with the Fubini-Study metric and $X^n\subset{\mathbb P}^N$ be endowed with the induced metric. We take $x\in X$ and still follow the notation and symbols introduced at the beginning of  Section \ref{reduction}. Denote by $N_xX\cong {\mathbb C}^{N-n}$ the orthogonal complement of $T_xX$ in $T_x{\mathbb P}^N$, and by $V^{\perp}$ the $N_xX$-component of any $V\in T_x{\mathbb P}^N$. The complex vector bundle $NX$ can be \emph{smoothly} identified with the normal bundle of $X$ which is the holomorphic quotient $T{\mathbb P}^N/TX$. The second fundamental form $II$ of $X\subset{\mathbb P}^N$  is the symmetric bilinear map
\be\begin{split}
II: \ \ \ & TX \times TX \longrightarrow NX\\
&(V,V') \longmapsto II(V, V'):=(\nabla_V V')^{\perp}
\end{split}\nonumber\ee
for any type $(1,0)$ tangent vector fields $V$, $V'$ in $X$, and $\nabla$ is the Levi-Civita connection of ${\mathbb P}^N$.

\begin{lemma}\label{second fundamental form lemma}
Around $x$ there exists a local tangent frame $\{e_1,\ldots,e_n\}$, and a local normal frame $\{e_{n+1},\ldots,e_N\}$ such that
\be\label{1} II(e_i,e_j)=\sum_{\alpha=n+1}^{N}f^{\alpha}_{ij}e_{\alpha},\ \ 1\leq i,j\leq n.\ee
Moreover $\{e_1,\ldots,e_N\}$ is orthonormal at $x$.
\end{lemma}

\begin{proof}
We still adopt the notation and symbols before. Write $\varepsilon_i:=\frac{\partial}{\partial z^i}$ ($1\leq i\leq N$).
Then under the natural tangent frame $\{ \varepsilon_1, \ldots \varepsilon_N\}$, the Fubini-Study metric around the origin $x$ is represented by the matrix (\cite[p. 174]{Zheng00})
\be\label{metric}g =(g_{i\bar{j}})=\frac{1}{\eta}I_N - \frac{1}{\eta^2} \overline{z}^t\,z;\qquad \eta:=1+\sum_{i=1}^N|z^i|^2, \ \ z=(z^1,\ldots,z^N).\ee
Its inverse matrix is given by
$$g^{-1} = \eta (I_N+ \overline{z}^t\,z).$$
Thus the matrix of the Levi-Civita connection under the frame $\{\varepsilon_1,\ldots,\varepsilon_N\}$ is
\be \nonumber
\theta = (\partial g) g^{-1} = - \frac{1}{\eta}\big[(\partial \eta)I_N + \overline{z}^t \, d\,z)\big],
\ee
or equivalently
\be\label{2}
\theta_{ab} = - \frac{1}{\eta}\big(\delta_{ab}\sum_{c=1}^N \overline{z^c}\,dz^c  +\overline{z^a}\,dz^b\big), \ \ \ \ \ \  \ \forall \ 1\leq a, b\leq N,
\ee
and so
\be\label{3}\nabla\varepsilon_a=\sum_{b=1}^N\theta_{ab}\varepsilon_b,\ \ \ \forall \ 1\leq a\leq N. \ee

Let
\be e_i:=\varepsilon_i+\sum_{\beta=n+1}^Nf^{\beta}_i\varepsilon_{\beta},\ \ \ e_{\alpha}:= \varepsilon_{\alpha}^{\perp}.\ \ \ \ (1\leq i\leq N, \ \ n+1\leq\alpha\leq N) \ee
Then $\{ e_1, \ldots , e_n\}$ form a local homomorphic tangent frame of $X$ and $\{e_{n+1},\dots,e_N\}$ become a local frame in the normal bundle $NX$. It is easy to see that $\{e_1,\ldots,e_N\}$ is orthonormal at $x$. It suffices to verify (\ref{1}), which is a routine calculation. Indeed, by adopting the index range (\ref{index range})
we have for any fixed $1\leq i,j\leq n$,
\begin{eqnarray*}
\nabla_{e_i}e_j&=&\sum_{\alpha,\beta}
\nabla_{(\varepsilon_i+f^{\alpha}_i\varepsilon_{\alpha})}
(\varepsilon_j+f^{\beta}_j\varepsilon_{\beta})\\
&=&\sum_{\alpha,\beta}\big[\nabla_{\varepsilon_i}\varepsilon_j+f^{\alpha}_i
\nabla_{\varepsilon_{\alpha}}\varepsilon_j+\varepsilon_i(f^{\beta}_j)\varepsilon_{\beta}+
f^{\beta}_j\nabla_{\varepsilon_i}\varepsilon_{\beta}+f^{\alpha}_i\varepsilon_{\alpha}(f^{\beta}_j)\varepsilon_{\beta}
+f^{\alpha}_if^{\beta}_j\nabla_{\varepsilon_{\alpha}}\varepsilon_{\beta}\big]\\
&=&\sum_{a,\beta}\big[\theta_{ja}(e_i)\varepsilon_a+f^{\beta}_{ij} \varepsilon_{\beta}+   f^{\beta}_j\theta_{\beta a}(e_i)\varepsilon_a\big]  \qquad\big(\text{by (\ref{3})}\big) \\
&=&\sum_{\beta}f^{\beta}_{ij} \varepsilon_{\beta}+\sum_{a,\beta} \big[\theta_{ja}(e_i)+f^{\beta}_j\theta_{\beta a}(e_i)\big]\varepsilon_a\\
& = &  \sum_{\beta}f^{\beta}_{ij} \varepsilon_{\beta} +\sum_{k,\beta}\big[\theta_{jk}(e_i) +f^{\beta}_j \theta_{\beta k}(e_i) \big] \varepsilon_k + \sum_{\alpha,\beta}\big[\theta_{j\alpha}(e_i) + f^{\beta}_j \theta_{\beta \alpha}(e_i) \big] \varepsilon_{\alpha} \\
& = &\sum_{\beta}f^{\beta}_{ij} \varepsilon_{\beta} -\frac{1}{\eta} (\eta_i \varepsilon_j +\eta_j \varepsilon_i)  -\sum_{\alpha}\frac{1}{\eta}\big[ f^{\alpha}_i\eta_j + f^{\alpha}_j\eta_i\big] \varepsilon_{\alpha} \ \ \Big(\text{by (\ref{2}), $\eta_i:=\overline{{z}^i}+\sum_{\alpha}f^{\alpha}_i\overline{{z}^{\alpha}}$}\Big)\\
& = & \sum_{\beta}f^{\beta}_{ij} \varepsilon_{\beta} - \frac{1}{\eta} (\eta_i e_j + \eta_j e_i).
\end{eqnarray*}
Therefore,
\be\nonumber
II(e_i,e_j) = \big( \nabla_{e_i}e_j \big)^{\perp} = \sum_{\alpha=n+1}^N f^{\alpha}_{ij} \,\varepsilon_{\alpha}^{\perp}=\sum_{\alpha=n+1}^N f^{\alpha}_{ij} \,e_{\alpha}.
\ee
\end{proof}

\section{The algebraic question}\label{algebraic question}
It turns out that a key factor in analyzing the codimension $N-n$ of the projective submanifold $X\subset {\mathbb P}^N$ is the dimension of the linear space spanned by $\{ H^{n+1}, \ldots , H^N\}$, under the restrictions (\ref{eq:width2}) and (\ref{eq:nondeg}). So we give the following
\begin{definition}\label{definition2}
\begin{enumerate}
\item
Fix positive integers $n\geq k\geq 2$. A set of symmetric $n\times n$ matrices ${\mathcal H}=\{ H^1, \ldots , H^r\}$ ($r$ can vary and $r\geq k$) is called an \emph{$(n,k)$-system} if it satisfies the \emph{width-$k$ condition}:
\be\label{width-k}H^{i_1}_u \wedge \cdots \wedge H^{i_k}_u = 0, \ \ \ \ \  \forall \ 1\leq i_1 < \cdots < i_k\leq r, \ \ \forall \  u \in {\mathbb C}^n,\ee
and the \emph{non-degeneracy condition}:
\be\label{non-degeneracy}\bigcap_{i=1}^r \ker (H^i)=0.\ee
\item
We define $l(n,k)$ to be
$$l(n,k):=\max\big\{\text{rank}(\mathcal{H})~|~\text{$\mathcal{H}$ are $(n,k)$-systems}\big\}.$$
\end{enumerate}
\end{definition}

With this definition in hand, the algebraic situation we are concerned with becomes the following

\begin{question}\label{question2}
What is the value of $l(n,k)$? When $l(n,k)\geq k$, for those $(n,k)$-systems ${\mathcal H}$ with $\text{rank}(\mathcal{H})\geq k$, what kind of special structure must they possess?
\end{question}

Clearly $l(n,k)\leq\frac{1}{2}n(n+1)$, which is the dimension of the space of all $n\times n$ symmetric matrices. On the other hand we have the following lower bound.

\begin{example}\label{examplelowerbound}
We have
\be\nonumber
l(n,k) \geq \frac{1}{2}(k-1)(k-2) + 1,\ \ \ \text{when $k\geq 4$}.
\ee
Indeed we can take a basis all symmetric $(k-2)\times (k-2)$ matrices by adding zero blocks at right and bottom to be our $\{ H^1, \ldots , H^m\}$, where $m=\frac{1}{2}(k-1)(k-2)$, and add on one more non-degenerate $n\times n$ matrix $H^{m+1}$. Note that $m\geq k-1$ when $k\geq 4$, and the first $m$ matrices satisfy the width-$(k-1)$ condition. So the whole $\{ H^1, \ldots , H^m, H^{m+1}\}$ satisfy the width-$k$ condition and it also satisfies the non-degeneracy condition as $\ker(H^{m+1})=\{0\}$, thus it is an $(n,k)$-system.
\end{example}

This example shows that, $l(n,k)\geq k$ when $k\geq 4$. So the real question is what can we say about an $(n,k)$-system ${\mathcal H}$ when $\text{rank}({\mathcal H})\geq k$?

For $k=2$, the answer is rather special: $l$ must be $1$ regardless of the value of $n$, namely, $l(n,2)=1$.
\begin{lemma} \label{codim1}
If two symmetric $n\times n$ matrices $H^1$ and $H^2$ satisfy the condition $H^1_u \wedge H^2_u =0$ for any $u\in {\mathbb C}^n$. Then they must be proportional. Namely, one is a constant multiple of the other.
\end{lemma}

This was proved at the end of \cite{LZ} as the algebraic component of the proof of Theorem 2.2 there, which characterizes hypersurfaces by the condition $\sigma_2(X,L)=0$.

For $k=3$, we have the following slightly more informative statement.
\begin{lemma}\label{codim2}
Let $\{ H^1, H^2, H^3\}$ be three symmetric $n\times n$ matrices that are linearly independent and $H^1_u \wedge H^2_u \wedge H^3_u=0$ for any $u\in {\mathbb C}^n$. Then the common kernel $\,\bigcap_{i=1}^3 \ker(H^i)$  is $(n-2)$-dimensional.
\end{lemma}

We will postpone the proof of this lemma to Appendix \ref{proof of codim2}. Write $H^w=\sum_{i=1}^3 a_iH^i$ for $w=(a_1, a_2, a_3)$. It is clear that
$$ \bigcap_{i=1}^3 \ker(H^i) = \bigcap_{w\in {\mathbb C}^3} \ker(H^w).$$
 So each $H^w$ has rank at most $2$. It is also easy to see that:

 \begin{remark} \label{remark}
Let $\{ H^1, H^2, H^3\}$ be as in Lemma \ref{codim2}. For any two linearly independent $\{w, w'\}$, the intersection of their kernels is already equal to $\,\bigcap_{i=1}^3 \ker(H^i)$.
 \end{remark}

As an immediate corollary, Lemma \ref{codim2} implies that in Question \ref{question2}, we have  $l(n,3)\leq2$.

\begin{corollary}\label{cor}
Let $\mathcal{H}=\{ H^1, H^2,\ldots ,H^r\}$ ($r\geq3$) be any $(n,3)$-system ($n\geq3$).  Then $\text{rank}(\mathcal{H})\leq 2$.
\end{corollary}

\begin{proof}
Suppose on the contrary $\text{rank}(\mathcal{H})\geq 3$.
Without loss of generality, let us assume that $\{ H^1, \ldots , H^l\}$ forms a basis of $\{ H^1, \ldots , H^r\}\,$  so $l\geq3$. By applying Lemma \ref{codim2} to  $\{ H^1, H^2, H^3\}$, we know that the space $K:=\bigcap_{i=1}^3\ker (H^i)$ is $(n-2)$-dimensional. Also, by Remark \ref{remark}, we know that
$\ker (H^1) \cap \ker (H^2) = K$.

For any $3\leq j\leq l$, by applying Lemma \ref{codim2} to $\{ H^1, H^2, H^j\}$, we get $ \ker (H^j) \supset K$ thus
$\bigcap_{i=1}^l\ker(H^i)=K$.

For any $j>l$, $H^j$ is a linear combination of $\{ H^1, \ldots , H^l\}$, hence $K=\bigcap _{i=1}^l \ker (H^i)\subset\ker(H^j)$. In summary, we have $\bigcap _{i=1}^r \ker (H^i)=K$ is $(n-2)$-dimensional and is nonzero since $n\geq3$. This contradicts to the non-degeneracy condition of $\mathcal{H}$.
\end{proof}

For $k=4$, we no longer have the luck of $l(n,k)<k$ as in the cases of $k=2$ and $k=3$. In this case $l(n,4)\geq 4$ as illustrated by Example \ref{examplelowerbound}.
Another example with $\text{rank}(\mathcal{H})\geq 4$ for $k=4$ is given by the second fundamental form of the Segre fourfold ${\mathbb P}^2 \times {\mathbb P}^2 \subset {\mathbb P}^8$. The four matrices $H^i$ are given by:
$$
\left[ \begin{array}{cccc} 0 & 0 & 1 & 0 \\  0 & 0 & 0 & 0 \\ 1 & 0 & 0 & 0 \\ 0 & 0 & 0 & 0  \end{array} \right] , \ \ \ \left[ \begin{array}{cccc} 0 & 0 & 0 & 1 \\  0 & 0 & 0 & 0 \\ 0 & 0 & 0 & 0 \\ 1 & 0 & 0 & 0  \end{array} \right] , \ \ \ \left[ \begin{array}{cccc} 0 & 0 & 0 & 0 \\  0 & 0 & 1 & 0 \\ 0 & 1 & 0 & 0 \\ 0 & 0 & 0 & 0  \end{array} \right] , \ \ \ \left[ \begin{array}{cccc} 0 & 0 & 0 & 0 \\  0 & 0 & 0 & 1 \\ 0 & 0 & 0 & 0 \\ 0 & 1 & 0 & 0  \end{array} \right] .
$$
Clearly they are linearly independent, and satisfy the width-$4$ condition as well as the non-degeneracy condition. Note that this computation of second fundamental form also indicates that the Chern form $C_4(\gamma^{\ast} Q)$ vanishes everywhere, thus giving an alternative proof of the fact that $\sigma_4(X,L)=c_4(\gamma^{\ast} Q)=0$.

For $k=4$, we do have the following positive answers.

\begin{lemma}\label{codim3}
Let $n\geq4$ and $\mathcal{H}=\{ H^1, \ldots , H^r\}$ ($r\geq4$) be an $(n,4)$-system. Namely, $\mathcal{H}$ satisfies the width-$4$ condition (\ref{width-k}) and the non-degeneracy condition (\ref{non-degeneracy}). Then
\begin{enumerate}
\item
$\text{rank}(\mathcal{H})\leq 4$.

\item
When $\text{rank}(\mathcal{H})=4$ and $n\geq 5$, we may assume that $\{H^1,\ldots,H^4\}$ are linearly independent. Replacing $\{ H^1, \ldots , H^4\}$ by another basis of $\mbox{Span} \{ H^1, \ldots , H^4\}$ if necessary, the first three matrices will lie in a ${\mathbb C}^2$. That is, there exists a nonsingular $n\times n$ matrix $A$ such that
$$ A\,H^i\,A^t = \left[ \begin{array}{ll} \ast & 0 \\ 0 & 0_{n-2} \end{array} \right] , \ \ \ i=1,2,3. $$
\end{enumerate}
\end{lemma}

In other words $l(n,4)=4$, and more importantly, when $n\geq 5$ and $\text{rank}(\mathcal{H})=4$, the system only comes with the above special structure. This special structure will be crucial for us in the proof of Theorem \ref{thm1.4}. Again we will postpone the proof of this algebraic lemma to the last two appendices: Appendices \ref{proof of codim3} and \ref{proof of codim3-2}.

\section{Proof of Theorem \ref{thm1.3}}\label{proof 1}
With the algebraic results in Section \ref{algebraic question} in hand, we are now able to prove in this section our first main result, Theorem \ref{thm1.3}.

First let us recall several notions in algebraic geometry, which play crucial roles in the proof. The {\em tangent variety} $\text{Tan}(X)$ and {\em secant variety} $\text{Sec}(X)$ of an $n$-dimensional projective submanifold $X$ in $\mathbb{P}^N$ are defined by
\be\label{c1c2}
\text{Tan}(X):=\bigcup_{x\in X}\widetilde{T}_x(X),\ \ \
\text{Sec}(X):=\overline{\{\text{lines $\overline{uv}$}~|~u,v\in X,~u\neq v\}},
\ee
whose maximal dimensions are $2n$ and $2n+1$, respectively. Here $\widetilde{T}_x(X)$ denotes the $n$-dimensional \emph{projective} tangent space of $X$ at $x$ introduced in (\ref{Gauss map}) and $``\overline{(\cdot)}"$ the Zariski closure.  $\text{Tan}(X)$ and $\text{Sec}(X)$ are both closed irreducible subvarieties in ${\mathbb P}^N$, with $\text{Tan}(X) \subset \text{Sec}(X)$.
The \emph{second osculating space} $\widetilde{T}^{(2)}_xX $ at $x$ is the span of the second osculating spaces at $x$ to all curves lying in $X$ ( \cite[p. 372]{GH}).

For our later convenience, several well-known facts in algebraic geometry related to the above notions are collected in the form of the following two lemmas.
\begin{lemma}\label{tangentsecantlemma}
\begin{enumerate}
\item
Let $X\subset {\mathbb P}^N$ be a projective manifold, then either $\dim\text{Tan}(X)=2n$ and $\dim\text{Sec}(X)=2n+1$, or else $\text{Tan}(X)=\text{Sec}(X)$.
\item
We have
\be\label{dimension of osculating space}
\dim\widetilde{T}^{(2)}_xX=n+\text{rank}\{H^{n+1},\ldots,H^N\}.
\ee
\item
For generic $y\in\text{Tan}(X)$ and $y\in\widetilde{T}_x(X)$ for generic $x$, we have
\be\label{subset relation}\widetilde{T}_y\text{Tan}(X)\subset\widetilde{T}_x^{(2)}(X).\ee
\end{enumerate}
\end{lemma}

\begin{proof}
The first part is a classical result due to Fulton and Hansen (\cite[Coro. 4]{FH79} or \cite[p. 215]{La}). The second part follows from (\ref{1}) and the relation between the second osculating space and the second fundamental form (\cite[(1.45)]{GH}). The third part can be directly checked by the definition (cf. \cite[Lemma 1]{BF}).
\end{proof}

The following result is a well-known fact in algebraic geometry. We include a proof here for the reader's convenience.

\begin{lemma}\label{projection}
For any non-degenerate embedding $X\subset {\mathbb P}^N$  associated to $L$, $\dim\text{Sec}(X)$  depends only on $(X,L)$ and not on the particular embedding. Moreover, $r_L=\dim\text{Sec}(X)-n$.
\end{lemma}
\begin{proof}
Let $i: X\hookrightarrow {\mathbb P}^N$ be the inclusion map, and let $s=\{ s_0, \ldots , s_N\}$ be the restriction on $X$ of a basis of $H^0({\mathbb P}^N, {\mathcal O}(1))$.  Extend $s$ to a basis $ s'= \{ s_0, \ldots , s_{N_0}\}$ of $H^0(X,L)$. Then $s'$ gives via the Kodaira map a non-degenerate embedding $i_0: X \hookrightarrow {\mathbb P}^{N_0}$. Let $P\cong {\mathbb P}^{N_0-N-1}$ be the linear subspace in ${\mathbb P}^{N_0}$ given by $\{ [0 : \cdots : 0 : \ast : \cdots : \ast ]\}$, where the first $N+1$ coordinate components are zero. Any point of $i_0(X)$ is not in $P$. Also, any line joining two points of $i_0(X)$ does not intersect $P$, so $S\cap P=\phi$, where $S$ is the secant variety of $i_0(X)$ in ${\mathbb P}^{N_0}$. Let $\pi : {\mathbb P}^{N_0}\setminus P \rightarrow {\mathbb P}^N$ be the projection map, then we have $i=\pi \circ i_0$, and the restriction of $\pi$ on $i_0(X)$ gives an isomorphism between $i_0(X)$ and $i(X)$. Clearly, $\pi $ also gives an isomorphism between $S$ and $\text{Sec}(X)$. In particular, $\dim \big( \text{Sec}(X) \big) = \dim S$, so it depends only on $X$ and $L$ but not on the particular embedding $i$ associated to $L$. Denote by $m:=\dim\text{Sec}(X)$. Then in ${\mathbb P}^N$, if we choose a linear subspace $P'\cong {\mathbb P}^{N-m-1}$ which does not intersect $\text{Sec}(X)$, then the restriction on $X$ of the projection map $\pi' : {\mathbb P}^N\setminus P' \rightarrow {\mathbb P}^m$ will give us an embedding $X\subset {\mathbb P}^m$ associated to $L$, and this is clearly the smallest codimension possible. So $r_L=m-n$ as claimed.
\end{proof}

Combining the above two lemmas, we get

\begin{corollary} \label{codimbound}
Let $X^n\subset {\mathbb P}^N$ be an embedding associated to $L$, and $l:=\text{rank}\{ H^{n+1}, \ldots , H^N\}$ at a generic point of $X$. If $l<n$, then $r_L\leq l$.
\end{corollary}
\begin{proof}
By (\ref{dimension of osculating space}) and (\ref{subset relation}) in Lemma \ref{tangentsecantlemma}, and our assumption, we have $$\dim \big( \text{Tan}(X)\big) \leq n+l<2n.$$
Hence $\text{Tan}(X)=\text{Sec}(X)$ by the first part in Lemma \ref{tangentsecantlemma}. With Lemma \ref{projection} we have $$r_L=\dim\text{Tan}(X)-n\leq l.$$
\end{proof}

This immediately gives us the proof of Theorem \ref{thm1.3}:

\noindent \begin{proof}[{\bf Proof of Theorem \ref{thm1.3}}]
Let $X^n$ be a projective manifold with $n\geq 3$ and $L$ a very ample line bundle on it with $\sigma_3(X,L)=0$. Let $X\subset {\mathbb P}^N$ be a non-degenerate embedding associated to $L$. If the codimension $N-n\leq 2$, then we are done.
So suppose on the contrary that $N-n\geq3$. In this case by Lemmas \ref{reduce algebraic lemma} and \ref{generic intersection} the set $\mathcal{H}=\{H^{n+1},\ldots,H^N\}$ forms an $(n,3)$-system at a generic point. By Corollary \ref{cor}, we know that $\text{rank}(\mathcal{H})\leq2<n$. So we get $r_L\leq 2$ by Corollary \ref{codimbound} above.
\end{proof}

Similarly by Lemma \ref{codim3} we get the following codimension upper bound in the case of $k=4$.

\begin{theorem}\label{codim4}
Let $X^n$ be a projective manifold with $n\geq 5$ and $L$ a very ample line bundle on it with $\sigma_4(X,L)=0$. Then $r_L\leq 4$. In particular, such an $X$ can be embedded in ${\mathbb P}^{n+4}$.
\end{theorem}

\begin{proof}
Let $X\subset {\mathbb P}^N$ be a non-degenerate embedding associated to $L$. If $N-n\geq4$, as above the set $\mathcal{H}$ forms an $(n,4)$-system at a generic point. We know by Lemma \ref{codim3} that $\text{rank}(\mathcal{H})\leq4<n$. So by Corollary \ref{codimbound} we get $r_L\leq 4$.
\end{proof}

\section{Proof of Theorem \ref{thm1.4}}\label{proof 2}
Suppose that $L$ is a very ample line bundle on a projective manifold $X^n$ of $n\geq 5$, satisfying $\sigma_4(X,L)=0$. By Theorem \ref{codim4}, there is an embedding $X\subset {\mathbb P}^{n+4}$ associated to $L$. Our goal is to show that $r_L\leq 3$.

\subsection{Preliminaries}
First we shall need the following relative version of the aforementioned Fulton-Hansen theorem, which is essentially due to Fulton-Hansen and Zak (\cite{Zak}), and a proof can be found in \cite[Thm 3.2.1]{Ru}.

\begin{lemma}\label{relative}
Let $X^n\subseteq\mathbb{P}^N$ be an irreducible projective variety of dimension $n$ and $Y\subset X$ a closed subvariety of dimension $n'$. Then either $\dim T^{\ast}(Y,X) =n'+n$ and $\dim S(Y,X)=n'+n+1$, or $T^{\ast}(Y,X)=S(Y,X)$.
\end{lemma}

Here $S(Y,X)$ is the {\em relative secant variety of $X$ with respect to $Y$}, defined as the Zariski closure of the union of all lines $\overline{xy}$ with $x\in X$, $y\in Y$, and $x\neq y$. $T^{\ast }(Y,X)$ is called the {\em relative tangent star of $X$ with respect to $Y$}, defined as the union of $T^{\ast }_y(Y,X)$ for all $y\in Y$, where
\be
T^{\ast }_y(Y,X) = \overline{ \{\lim_{u\rightarrow y,\, x\rightarrow y}\text{lines $\overline{ux}$}~|~u \in Y,\,x\in X,~u\neq x\} }
\ee
In particular, $T^{\ast}_y(\{y\},X)=C_yX$ is the {\em tangent cone} of $X$ at $y$, and $T^{\ast}_y(X,X)=T^{\ast}_yX$ is called the {\em tangent star} of $X$ at $y$. One always has $C_yX\subset T^{\ast}_yX \subset \widetilde{T}_yX$, where the last term is called the {\em tangent space} of $X$ at $y$, which is the smallest linear subspace in ${\mathbb P}^N$ containing the tangent cone. When $X$ is smooth at $y$, one has $C_yX= T^{\ast}_yX = \widetilde{T}_yX$.

Similarly, $T(Y,X)=\cup_{y\in Y}\widetilde{T}_yX$ is called the {\em relative tangent variety of $X$ with respect to $Y$}.

Write $N=n+4$. We may assume that the rank $l(x)$ of the second fundamental form at a generic point $x\in X$ is equal to $4$, as otherwise we would have $r_L\leq 3$ already. Denote by $X'$ the open dense subset of $X$ where $l(x)=4$. By the second part of Lemma \ref{codim3}, we get a special structure about the second fundamental form at the points in $X'$.

Fix a point $x\in X'$. With Lemma \ref{second fundamental form lemma} in mind, we may assume that $\{e_1,\ldots,e_n\}$  and $\{e_{n+1}, \ldots , e_{n+4}\}$ be unitary frame of the tangent space $T_xX$ and the normal space $N_xX$ in ${\mathbb P}^N$ at $x$, and by an abuse of notation, denote by $H^{\alpha}$ the $e_{\alpha}$-component of the second fundamental form:
$$H^{\alpha}(\cdot,\cdot):=\langle II(\cdot,\cdot), e_{\alpha}\rangle ,\qquad n+1\leq\alpha\leq n+4=N.$$
It is a symmetric bilinear form on the tangent space $T_xX$. For simplicity we will still denote by $H^{\alpha}$ the symmetric $n\times n$ matrix under the basis $\{e_1,\ldots,e_n\}$ and $H^{\alpha}_{ij}:=H^{\alpha}(e_i,e_j)$.  In terms of $H^{\alpha}$, the second part of Lemma \ref{codim3} implies that there exist subspace $N'_x \subset N_xX$ and $F_x\subset T_xX$, with $N'_x \cong {\mathbb C}^3$ and $F_x\cong {\mathbb C}^{n-2}$, such that $F_x$ is the common kernel of $H^w$ for all $w\in N'_x$.

Here we have adopted the notation that
$H^w :=\sum_{i=1}^4 w_i H^{n+i}$
for $w=\sum_{i=1}^4 w_ie_{n+i}$. Note that the subspace $N'_x\subset N_xX$ is uniquely determined as the set of all $w\in N_xX$ such that $\text{rank}(H^w)$ is at most $2$, by the fact $n\geq 5$ and the non-degeneracy condition (\ref{eq:nondeg}). So $F_x$, as the common kernel of $H^w$ for all $w\in N'_x$,  is also uniquely determined. In the open dense subset $X'$ of $X$, $F$ forms a distribution.

\begin{lemma} \label{foliation}
$F$ is a holomorphic foliation.
\end{lemma}
\begin{proof}
With the above notation in mind,
let $\{ e_1, \ldots , e_{n}\}$ be a local unitary frame of $X$, and $\{ e_{n+1}, \ldots , e_{n+4}\}$ be a local unitary frame for the normal bundle, such that $\{e_3, \ldots , e_{n}\}$ spans $F$ and $\{ e_{n+1}, \ldots , e_{n+3}\}$ spans $N'$ at each point. We may extend this local frame along $X$ to a local unitary tangent frame of an open subset in ${\mathbb P}^N$ ($N=n+4$), and denote by $\theta $, $\Theta$ the matrix of connection and curvature of ${\mathbb P}^N$ under the frame $\{e_1,\ldots,e_N\}$. Let $\varphi_a$ ($1\leq a\leq N$) be the coframe dual to $e_a$. Then we have
$$ d\varphi = -\theta^t \wedge \varphi , \ \ \ \Theta = d\theta - \theta \wedge \theta.$$
Since $\{e_a\}$ is unitary, we have
 \be \label{eq:Theta}
 \begin{split}
 \Theta_{ab} = \sum_{c,d=1}^NR_{c\overline{d}a\overline{b} }\,\varphi_c \wedge \overline{\varphi_d}& =  \sum_{c,d=1}^N (\delta_{cd}\delta_{ab} + \delta_{cb}\delta_{ad})\,\varphi_c \wedge \overline{\varphi_d}\\
  &=    \delta_{ab}\sum_{c=1}^N \varphi_c \wedge \overline{\varphi_c} +\varphi_b \wedge \overline{\varphi_a}.\end{split}
 \ee
Let us fix the index range throughout the proof of this lemma:
\be
1\leq i,j\leq n; \ \ \ n+1\leq \alpha, \beta \leq N=n+4; \ \ \ 1\leq a, b\leq N=n+4.
\ee
Restricted on (an open subset of) $X$, which is defined by $\varphi_{\alpha}=0$ for all $\alpha$, we know that  $\theta_{i\alpha}$ are $(1,0)$-forms and give the second fundamental form
$$ \theta_{i\alpha} = \sum_j H^{\alpha}_{ij} \varphi_j. $$
By our construction of $N'$ and $F$, we know that
\be
\theta_{i\alpha} =0 \ \ \ \ \ \ \forall \ i=3,\ldots,n, \ \ \forall \ \alpha=n+1,n+2,n+3,
\ee
while
\be
\theta_{i\alpha}= \sum_{j=1}^2 H^{\alpha}_{ij} \varphi_j, \ \ \ \ \forall i=1,2, \ \ \forall \ \alpha=n+1,n+2,n+3.
\ee
Let us fix $i>2$ and $\alpha <n+4$. By (\ref{eq:Theta}), we have $\Theta_{i\alpha}=0$ since $\varphi_{\alpha}=0$, and so
\begin{eqnarray*}
0 & = & \Theta_{i\alpha} \ = \ d\theta_{i\alpha} - \sum_a \theta_{ia} \theta_{a\alpha} \ = \  - \sum_a \theta_{ia} \theta_{a\alpha}  \\
& = &  - \sum_j \theta_{ij} \theta_{j\alpha} - \sum_{\beta} \theta_{i\beta } \theta_{\beta \alpha} \ = \  - \sum_{j=1}^2 \theta_{ij} \theta_{j\alpha} - \theta_{iN } \theta_{N \alpha}
\end{eqnarray*}
So we get
\be \label{eq:forms}
\sum_{j=1}^2 (H^{\alpha}_{1j}\varphi_1 + H^{\alpha}_{2j}\varphi_2)\wedge \theta_{ij}= (H^N_{i1}\varphi_1 + \cdots + H^N_{in}\varphi_n)\wedge \theta_{N\alpha}
\ee
for each $i>2$ and $\alpha <N$. Let us write $\psi = \psi' + \psi''$ for the decomposition of a $1$-form into its $(1,0)$ and $(0,1)$ parts.  Taking the $(1,1)$-part in (\ref{eq:forms}), we get
$$
\sum_{j=1}^2 (H^{\alpha}_{1j}\varphi_1 + H^{\alpha}_{2j}\varphi_2)\wedge \theta_{ij}'' = (H^N_{i1}\varphi_1 + \cdots + H^N_{in}\varphi_n)\wedge \theta_{N\alpha}''
$$
If $H^N_{ik}=0$ for all $i,k>2$, then the common kernel $\bigcap_{\beta =n+1}^N \ker(H^{\beta})\neq 0$ as $n>4$, so by the above identity we must have $\theta_{N\alpha}''=0$, and then $\theta_{i1}''=\theta_{i2}''=0$ since $H^w$ has rank $2$ for generic $w\in N'$. This means that $\nabla_{\overline{Z}}F \subset F$ for any $(1,0)$-type vector field $Z$, so $F$ is holomorphic.

Next, for any $(1,0)$-form $\psi$, let us denote by $\tilde{\psi}$ the part modulo $\{ \varphi_1, \varphi_2\}$. That is, if $\psi$ is given  by $\sum_{j=1}^n a_j\varphi_j$, then  $\tilde{\psi} = \sum_{j=3}^n a_j\varphi_j$. Modulo $\{ \varphi_1, \varphi_2\}$ in  (\ref{eq:forms}), we get $\tilde{\theta}_{iN} \wedge \tilde{\theta}_{N\alpha } = 0$. So if $\tilde{\theta}_{N\alpha }\neq 0$, then it will force the lower right $(n-2)\times (n-2)$ block of $H^N$ to have rank at most $1$. This will lead to a nonzero element in the common kernel of $H$ if $n\geq 6$, contradicting with the non-degeneracy condition.

So when $n\geq 6$ we must have $\tilde{\theta}_{N\alpha}=0$ for any $\alpha <N$. Let us write  $\theta_{N\alpha} = p_{\alpha} \varphi_1 + q_{\alpha} \varphi_2$. Also write $H^{\alpha}_{11}=a_{\alpha}$, $H^{\alpha}_{12}=b_{\alpha}$, and $H^{\alpha}_{22}=c_{\alpha}$.  Formula (\ref{eq:forms}) leads to
$$
\theta_{1\alpha} \tilde{\theta}_{i1} + \theta_{2\alpha} \tilde{\theta}_{i2} = \tilde{\theta}_{iN} \theta_{N\alpha}.
$$
Or equivalently,
\begin{eqnarray*}
a_{\alpha} \tilde{\theta}_{i1} + b_{\alpha} \tilde{\theta}_{i2} & = &  - p_{\alpha} \tilde{\theta}_{iN} \\
b_{\alpha} \tilde{\theta}_{i1} + c_{\alpha} \tilde{\theta}_{i2} & = &  - q_{\alpha} \tilde{\theta}_{iN}
\end{eqnarray*}
Since $H^w$  has rank $2$ for generic $w\in N'$ by Remark \ref{remark}, the above equations lead to
\be
\tilde{\theta}_{i1} = \lambda \tilde{\theta}_{iN}, \ \ \ \tilde{\theta}_{i2} = \mu \tilde{\theta}_{iN}
\ee
for some functions $\lambda$ and $\mu$, independent of $i>2$. Now by the structure equation, modulo $\{\varphi_1, \varphi_2\}$, we have
$$ - d\varphi_1 =  \sum_{i=1}^n \theta_{i1}\varphi_i \equiv
\sum_{i>2}\tilde{\theta}_{i1}\varphi_i
=  \sum_{i>2}\lambda\tilde{\theta}_{iN}\varphi_i \equiv
\lambda \sum_{i,j>2} H^N_{ij}\varphi_i\varphi_j = 0.
$$
Similarly, $d\varphi_2 \equiv 0$ modulo $\{ \varphi_1, \varphi_2\}$. So $F$ is a foliation. Next let us consider the case $n=5$. In this case we no longer always have $\tilde{\theta}_{N\alpha}=0$ for all $n+1\leq \alpha \leq n+3$. Instead, by (\ref{eq:forms}), we get
\be \label{eq:4.11}
 \tilde{\theta}_{iN} \wedge \tilde{\theta}_{N\alpha} =0, \ \ \ 3\leq i\leq n, \, n<\alpha <N
 \ee
We may assume that not all $\tilde{\theta}_{N\alpha}$ are zero, otherwise the proof for the $n\geq 6$ case above will show that $F$ is a foliation. Also, there must be $i>2$ such that $\tilde{\theta}_{iN}\neq 0$, as otherwise the lower right $(n-2)\times (n-2)$ block of $H^N$ is zero, contradicting with the non-degeneracy condition since $n\geq 5$. By (\ref{eq:4.11}), we know that
$$ H^N = \left[ \begin{array}{ccc} \ast & \ast & x^t \\ \ast & \ast & y^t \\ x & y & zz^t \end{array} \right] , $$
where $x$, $y$, $z$ are column vectors in ${\mathbb C}^3$ with $x\wedge y \wedge z\neq 0$, while
$$ \theta_{N\alpha} = p_{\alpha} \varphi_1 + q_{\alpha} \varphi_2 + r_{\alpha} \psi, \ \ \ \ \theta_{iN} = x_i\varphi_1+y_i\varphi_i+ z_i \psi, \ \ \ \ \, \text{where}\ \
\psi = \sum_{i=3}^5 z_i\varphi_i.$$
Plug into (\ref{eq:forms}), we get
$$ \theta_{1\alpha} \tilde{\theta}_{i1} +  \theta_{2\alpha} \tilde{\theta}_{i2} = (x_ir_{\alpha}-z_ip_{\alpha})\varphi_1\psi + (y_ir_{\alpha}-z_iq_{\alpha})\varphi_2\psi, $$
or equivalently,
\be \label{eq:4.12}
\left\{ \begin{split}
& a_{\alpha} \tilde{\theta}_{i1} + b_{\alpha} \tilde{\theta}_{i2} = (x_ir_{\alpha} - z_i p_{\alpha} )\psi \\
& b_{\alpha} \tilde{\theta}_{i1} + c_{\alpha} \tilde{\theta}_{i2} = (y_ir_{\alpha} - z_i q_{\alpha} )\,\psi
\end{split} \right.
\ee
By a unitary rotation of $\{ e_{n+1}, e_{n+2}, e_{n+3}\}$ if necessary, we may assume that $r_{\alpha}=0$ for $\alpha = n+1$ and $n+2$. By Lemma \ref{codim3} and Remark \ref{remark}, we know that $H^w$ has rank $2$ for generic $w\in \text{Span} \{ e_{n+1}, e_{n+2}\}$, so by applying (\ref{eq:4.12}) for $\alpha = n+1$ and $n+2$, we get
$$ \tilde{\theta}_{i1} = \lambda z_i \psi, \ \ \ \tilde{\theta}_{i2} = \mu  z_i \psi $$
for $\lambda$, $\mu$ independent of $i$. Thus modulo $\{ \varphi_1, \varphi_2\}$, we have
$$ d\varphi_1 \equiv -  \tilde{\theta}_{i1}\wedge \varphi_i = - \lambda z_i \psi \wedge \varphi_i = -\lambda \psi \wedge (z_i\varphi_i) =  -\lambda \psi \wedge \psi =0.$$
Similarly, $d\varphi_2 \equiv 0$, so $F$ is a foliation. This completes the proof of the lemma.
\end{proof}

Note that in the case $n=5$ we also have $$ \tilde{\theta}_{i1} = \lambda \tilde{\theta}_{iN}, \ \ \ \tilde{\theta}_{i2} = \mu \tilde{\theta}_{iN} $$
for $\lambda$, $\mu $ independent of $i$. Since $\Theta_{iN}=0$ by (\ref{eq:Theta}), so we have
$$ d\theta_{iN}= \theta_{i1}\theta_{1N} +\theta_{i2}\theta_{2N} + \sum_{j>2} \theta_{ij}\theta_{jN} + \theta_{iN}\theta_{NN}. $$
Therefore,
\be
d\psi \equiv 0 \ \ \text{mod} \ \{ \varphi_1, \varphi_2, \psi\}
\ee
So each leaf $Y$ of $F$ is foliated by hypersurfaces $Z$ defined by $\psi =0$.

\begin{lemma} \label{constantT2}
Let $Y$ be the leaf of $F$ passing through a generic point $x$. Then either $\widetilde{T}^{(2)}_yY=P$ is constant for all $y\in Y$,  or $Y$ is holomorphically foliated by hypersurfaces, and along each hypersurface $Y_1$, $\widetilde{T}^{(2)}_yY=P$  is constant for all $y\in Y_1$. In both cases, $P\cong {\mathbb P}^{n-1}$ is a linear subspace in ${\mathbb P}^N$, and $P$ does not contain $Y\cap \Omega$ for any small neighborhood $\Omega$ of $x$.
\end{lemma}

\begin{proof}
 Along $Y$, we have $\varphi_1=\varphi_2=0$, and the second fundamental form of $Y\subset {\mathbb P}^N$ is given by $\tilde{\theta}_{iN}$, $\tilde{\theta}_{i1}=\lambda \tilde{\theta}_{iN}$,  $\tilde{\theta}_{i2}=\mu \tilde{\theta}_{iN}$, and $\theta_{i\alpha}=0$ for $n+1\leq \alpha \leq n+3$. Write
 \be \label{eNtilde}
  \tilde{e}_1 = e_1 - \lambda e_N, \ \ \ \tilde{e}_1 = e_2 - \mu e_N, \ \ \ \tilde{e}_N = \overline{\lambda} e_1 + \overline{\mu } e_2 + e_N.
 \ee
 Then $\{ \tilde{e}_1, \tilde{e}_2 , \tilde{e}_N \}$ form a basis of $\text{Span}\{e_1, e_2, e_N\}$, with $\tilde{e}_N$ perpendicular to $\tilde{e}_1$ and $\tilde{e}_2$.
 The normal bundle $NY$ splits as the orthogonal sum of  $N'Y:=\text{Span}\{ e_{n+1}, e_{n+2}, e_{n+3}, \tilde{e}_1, \tilde{e}_2 \}$ and ${\mathbb C}\,\tilde{e}_N$, and its second fundamental form is trivial in $N'Y$ directions but non-trivial in the $\tilde{e}_N$ direction,  as
 $ \tilde{\theta}_{i\tilde{1}}= \tilde{\theta}_{i1} - \lambda \tilde{\theta}_{iN}=0$ and similarly  $\tilde{\theta}_{i\tilde{2}} =0$, while $\tilde{\theta}_{i\tilde{N}} = (1+|\lambda|^2+|\mu |^2) \tilde{\theta}_{iN}\neq 0$.

 So the second fundamental form of $Y$ forms an one-dimensional space, and the second osculating space $\widetilde{T}_y^{(2)}Y$ is the linear space of dimension $n-1$ in ${\mathbb P}^N$ passing through $y$ and containing the directions of
 \be \label{T2Y}
 \text{Span} \{ e_3, \ldots, e_n, \tilde{e}_N \}
 \ee
 at $y$. Let $\gamma^{(2)}$ be the second Gauss map of $Y$, sending a point $y\in Y$ to the second osculating space of $Y$ at $y$. By \cite[(1.69)-(1.73)]{GH}, the image of $\gamma^{(2)}$ is at most one-dimensional. If the image is zero-dimensional, namely, $\gamma^{(2)}$ is a constant map, then $\widetilde{T}_y^{(2)}Y$ is constant for all $y\in Y$, and denote by $P=\widetilde{T}_x^{(2)}Y \cong {\mathbb P}^{n-1}$ this linear subspace. It contains the tangent variety $\text{Tan}(Y)$ of $Y$, hence is equal to (the closure of) $\text{Tan}(Y)$, since the latter has dimension $n-1$.

If the map $\gamma^{(2)}$ has one-dimensional image, then its fibers will foliate $Y$ into hypersurfaces. Let $Y_1$ be a generic fiber. Then $\widetilde{T}_y^{(2)}Y$ remains constant for all $y\in Y_1$, and the second case of the lemma occurs. Again write $P=\widetilde{T}_x^{(2)}Y \cong {\mathbb P}^{n-1}$ for this linear space. Note that in any neighborhood $\Omega$ of $x$, $P$ cannot contain $Y\cap \Omega$, as otherwise it will force $\gamma^{(2)}$  to be a constant map.  This completes the proof of the lemma.
\end{proof}

The following statement is more or less obvious, and we include a proof here for the sake of completeness.
\begin{lemma}\label{obvious lemma}
Let $P\subset Q$ be linear subspaces in ${\mathbb P}^N$ of dimensions $n$ and $m$ respectively. Then the tangent bundle $TQ$ is parallel along $P$.
\end{lemma}

\begin{proof}
Let $Z=[Z_0:\cdots :Z_N]$ be a unitary homogeneous coordinate of ${\mathbb P}^N$ such that
$$ P = \{ [Z_0: \cdots :Z_n: 0: \cdots : 0]\}, \ \ \ \ \ Q= \{ [Z_0: \cdots :Z_m: 0: \cdots : 0]\}. $$
In the open chart $U=\{ Z_0\neq 0\}$, $(z^1, \ldots , z^N)$ becomes holomorphic coordinates and $\{ \varepsilon_1, \ldots , \varepsilon_N\}$ becomes a tangent frame, where $z^i=\frac{Z_i}{Z_0}$ and $\varepsilon_i = \frac{\partial}{\partial z^i}$. As we have seen in (\ref{metric}) and (\ref{2}), under the frame $\varepsilon$, the entries of the matrices of Fubini-Study metric and Levi-Civita connection are
$$ g_{i\overline{j}} = \frac{1}{1+|z|^2} \delta_{ij} - \frac{1}{(1+|z|^2)^2} \overline{z^i} z^j,\ \  \theta_{ij} = -\frac{1}{1+|z|^2} \big( \delta_{ij} \sum_{k=1}^N \overline{z^k} dz^k + \overline{z^i} dz^j \big) ,  \ \ 1\leq i,j\leq N.$$
In our setting, $P$ is defined in $U$ by $z^{\alpha}=0$, $n+1\leq \alpha \leq N$, and $\{\varepsilon_{1}, \ldots , \varepsilon_n \}$ forms a tangent frame of $P$ and  $\{ \varepsilon_{n+1}, \ldots , \varepsilon_N\}$ forms a unitary normal frame of $P$. For any $n+1\leq \alpha \leq m$, any $m<\alpha' \leq N$, and any $1\leq i,j\leq n$, we have
$$ \theta_{\alpha \alpha'}(\varepsilon_j) = \theta_{i\alpha'}(\varepsilon_j) = 0.$$
So the bundle $TQ=\text{span}\{ \varepsilon_1, \ldots , \varepsilon_m\}$ is parallel along $P$.
\end{proof}

Since $Q$ is a linear space in ${\mathbb P}^N$, we have  $\widetilde{T}_yQ=Q$ for any $y\in Q$.

\subsection{The completion of the proof}
Now we are ready to finish the proof of Theorem \ref{thm1.4}.

\begin{proof}
Let $X\subset {\mathbb P}^N$ be a non-degenerate embedding associated to $L$, such that $n\geq 5$, $N=n+4$, and $\sigma_4(X,L)=0$.  Assume that the rank $l(x)$ of the second fundamental form of $X$ at a generic point $x$ is equal to $4$. Then an open dense subset $X'\subset X$ admits a holomorphic foliation $F$.  Let $Y$ be the leaf of $F$ passing through $x$. For any $y\in Y$, denote by $P_y$ the linear subspace in ${\mathbb P}^N$ passing through $y$ and containing the directions $\{ e_3, \ldots , e_n, \tilde{e}_N \}$ at $y$, which is just $\widetilde{T}_y^{(2)}Y$ by (\ref{eNtilde}) and (\ref{T2Y}).

By Lemma \ref{constantT2}, we know that  $P_y=P_x=P$ with $P\cong {\mathbb P}^{n-1}$ either for all $y\in Y$, or for all $y\in Y_1$, a smooth hypersurface in $Y$ which is a generic fiber of the second Gauss map of $Y$.

Write $E_y=\text{Span}\{ e_1, \ldots , e_n, e_N\}$ and $E^{\perp}_y = \text{Span}\{ e_{n+1}, e_{n+2}, e_{n+3}\}$ for bundles over $P$.

First consider the case when $\tilde{\theta}_{N\alpha} =0$ for all $n+1\leq \alpha \leq n+3$. In this case, since $\tilde{\theta}_{1\alpha} = \tilde{\theta}_{2\alpha} = \tilde{\theta}_{N\alpha} =0$, we see that $E^{\perp}$, hence $E$, is parallel along $Y$.

Now let $Q\cong {\mathbb P}^{n+1}$ be the linear subspace in ${\mathbb P}^N$ passing through $x$ containing the directions $\{ e_1, \ldots , e_n, e_N\}$ at $x$. Then $P\subset Q$ and so $TQ$ is parallel along $P$ by Lemma \ref{obvious lemma}.

By Lemma \ref{constantT2}, either $Y\subset P$ or $Y_1\subset P$, and in the latter case $Y$ cannot be contained in $P$ in any small neighborhood of $x$.

Let $Y'$ be the irreducible component of $P\cap X$ passing through $x$. Then $Y'=\overline{Y}$ in the  first case while $Y'=\overline{Y_1}$ in the second case. Suppose we are in the first case. Since $TQ$ is parallel on $P$, it is also parallel on $Y$. Now both $E$ and $TQ$ are parallel bundles over $Y$, and $E_x=T_xQ$. So for any $y\in Y$, we have $E_y=T_yQ$, hence
$T_yX \subset E_y = T_yQ$, which leads to $\widetilde{T}_yX \subset \widetilde{T}_yQ = Q$ for all $y\in Y$ hence for all $y\in Y'$ by taking the limit. Therefore,  $T^{\ast }(Y', X)\subset Q$, so by by Lemma \ref{relative} we get $S(Y', X)\subset Q$ hence $X\subset Q$, a contradiction to the assumption that $X\subset {\mathbb P}^N$ is non-degenerate. In the second case we get the same contradiction by using $Y_1$ instead of $Y$.

Next let us assume that not all $\tilde{\theta}_{N\alpha} =0$. By the discussion right after the proof of Lemma \ref{foliation}, we know that $Y$ is foliated by holomorphic hypersurfaces defined by $\psi=0$. Let $Z$ be such a hypersurface and consider $\gamma^{(2)}|_Z$, the restriction on $Z$ of the second Gauss map of $Y$. It will once again has image of dimension either $0$ or $1$. Note that $E^{\perp}$, hence $E$, is parallel along $Z$, hence the same argument as before on $Z$ or $Z_1$, a generic fiber of $\gamma^{(2)}|_Z$, would lead to a contradiction to the non-degenerateness of $X$. This shows that the assumption $l(x)=4$ actually cannot occur when $n\geq 5$, so $r_L\leq l(x)\leq 3$ and we have completed the proof of Theorem \ref{thm1.4}.
\end{proof}

\appendix

\section{Proof of Lemma \ref{codim2}}\label{proof of codim2}

First let us recall and fix some notations, which shall be used throughout this and the next two appendices. Let $V\cong {\mathbb C}^n$ be a complex vector space, ${\mathcal H}=\{ H^1, \ldots , H^r\}$ be a set of  quadratic (i.e., symmetric bilinear) forms on $V$. For $u\in V$, denote by $H^i_u:=H^i(u, \cdot)$ the element in the dual space $V^{\ast}$, and by $\ker (H^i):=\{ u\in V \mid H^i_u=0\}$ the kernel of $H^i$. This set ${\mathcal H}$ is said to satisfy the {\em non-degeneracy condition} if $\,\bigcap_{i=1}^r\ker (H^i)=0$, and when $r\geq k$, to satisfy the {\em width-$k$ condition} if
$H^{i_1}_u, \cdots, H^{i_k}_u$  are linearly dependent
in $V^{\ast}$ for any $u\in V$ and any $1\leq i_1< \cdots < i_k\leq r$.

Under any basis  $\{ e_1, \ldots , e_n\}$ of $V$, each $H^i$ is represented by a symmetric $n\times n$ matrix, which we still denote by $H^i$ when there is no danger of confusion. In this situation, all the relevant notations above are compatible with those in Definition \ref{def}. 

Let $W={\mathbb C}^r$, and write $H^w=\sum_{i=1}^r c_i H^i$ for any $w=(c_1, \ldots , c_r)$. It is easy to see that the non-degeneracy and width-$k$ conditions are invariant when replacing ${\mathcal H}$ by another basis of the span of these $r$ quadratic forms. In what follows we will call this a {\em scrambling} of these $H^i$. Note that for any fixed $H^w$, we can always choose a suitable basis of $V$ so that $H^w$ is represented by the matrix
\be \label{diag}
H^w = \left[ \begin{array}{ll} I_p & 0 \\ 0 & 0\end{array} \right],\ \ \  p=\text{rank}(H^{w}).
\ee

We begin with the following
\begin{lemma} \label{hyperplane}
Let $H$ be a quadratic form on $V\cong {\mathbb C}^n$. For any hyperplane  $V'$  in $V$, denote by  $\widetilde{H}:=H|_{V'\times V'}$ the restriction. Then for generic choice of a hyperplane $V'$,  it holds that $\ker (\tilde{H}) = \ker (H) \cap V'$.
\end{lemma}

\begin{proof}
Clearly, we always have $\ker (\widetilde{H}) \supset  \ker (H) \cap V'$ for any hyperplane $V'$. It suffices to show that \be\label{6}\dim\ker (\widetilde{H})\leq\dim\big(\ker (H) \cap V'\big),\ \ \text{for generic $V'$}.\ee
 To see this, let us take a basis $\{e_i\}$ of $V$ so that $H$ is in the block diagonal form (\ref{diag}) under $\{e_i\}$. Suppose that $V'$ is spanned by $n-1$ vectors $v_i=\sum_{j=1}^n A_{ij}e_j$, $1\leq i\leq n-1$. Then under the basis $\{v_i\}$ of $V'$, $\widetilde{H}$ is represented by the matrix $AH\, A^t=B\,B^t$, where we wrote $A=(B,C)$ and $B$ is the left $(n-1)\times p$ block of the $(n-1)\times n$ matrix $A$, where $p=\text{rank}(H)$.

If $p=n$, then $B=A$ and $\widetilde{H}=A\, A^t$. In this case the equality in the lemma holds when $\widetilde{H}$ is non-degenerate, or equivalently when the matrix $A\, A^t$ is non-degenerate. This is clearly the case for generic choice of $A$.

If $p<n$, then for generic choice of $A$, the $(n-1)\times (n-1)$ matrix $BB^t$ will have rank $p$, which is the maximum possible value. So $\ker(\widetilde{H})$ has dimension $n-1-p$, while the dimension of $\ker (H)\cap V'$ is at least $n-p-1$. This completes (\ref{6}) and thus the proof of the lemma.
\end{proof}

\begin{proof}[{\bf Proof of Lemma \ref{codim2}.}]
Assume that the conclusion fails. Then there exist integers $\hat{n}\geq 4$, $n\geq3$ and a linearly independent set $\{Q^1, Q^2, Q^3\}$ of quadratic forms on some $\hat{V} \cong {\mathbb C}^{\hat{n}}$ such that $Q^i$ satisfy the width-$3$ condition and their common kernel $K=\bigcap_{i=1}^3 \ker (Q^i)$ has codimension $n$.

Let $V\cong {\mathbb C}^n$ be a linear subspace such that $\hat{V}=V\oplus K$, and let $H^i$ be the restriction of $Q^i$ on $V$.
Then it is easy to see that the quadratic forms $H^1, H^2, H^3$ on $V\cong {\mathbb C}^n$ satisfy the the width-$3$ condition as well as the non-degeneracy condition. Let us assume that $n\geq 3$ is the \emph{smallest} such numbers, namely, such a set does not exist on any $V$ with dimension between $3$ and $n-1$. We want to derive at a contradiction.

We will divide the discussion into two cases, depending on $n>3$ or $n=3$.

{\bf Case 1: $n>3$.}

In this case, let us choose a \emph{generic} hyperplane $V'\cong {\mathbb C}^{n-1}$ in $V$, and consider the set of restriction quadratic forms $\{ \widetilde{H}^1, \widetilde{H}^2, \widetilde{H}^3\}$ on $V'$. It clearly satisfies the width-$3$ condition, and by Lemma \ref{hyperplane}
$$ \bigcap_{i=1}^3 \ker (\widetilde{H}^i) = \bigcap_{i=1}^3 (\ker (H^i)\cap V') = (\bigcap_{i=1}^3 \ker (H^i)) \cap V' = 0,$$
so it satisfy the non-degeneracy condition as well. Since $n-1\geq 3$, by our assumption on the \emph{minimality} of $n$,  the set $\{ \widetilde{H}^1, \widetilde{H}^2, \widetilde{H}^3\}$ must be linearly \emph{dependent}. Replacing $\{H^1, H^2, H^3\}$ by another basis of the spanning space if necessary, we may assume that $\widetilde{H}^1=0$. Choosing a basis $\{e_i\}$ of $V$ so that $V'$ is spanned by $\{ e_2, \ldots , e_n\}$, we have
\be \label{threematrix}
H^1= \left[ \begin{array}{ll} \lambda & x^t \\ x & 0 \end{array} \right], \ \ \ \ H^2= \left[ \begin{array}{ll} a & y^t \\ y & A \end{array} \right], \ \ \ \ H^3= \left[ \begin{array}{ll} b & z^t \\ z & B \end{array} \right],
\ee
where $x$, $y$, $z$ are column vectors in ${\mathbb C}^{n-1}$ and $A$, $B$ are $(n-1)\times (n-1)$ matrices. To streamline the writings, let us divide the discussion into three subcases, depending on the vanishing of $\lambda$ and $a$, $b$.

{\bf Subcase 1a: $\lambda \neq 0$.}

In this case, if we replace $ H^i$ by the new basis $\{ H^1, H^2-\frac{a}{\lambda}H^1, H^3-\frac{b}{\lambda}H^1\}$, which we shall frequently call a scrambling of $H^i$, we may \emph{assume that $a=b=0$}. For any column vector $u \in V\cong\mathbb{C}^n$ in the form $u=\binom{t}{v}$ where $v\in V'$,  we have
\be\label{threecolumn}
H^1_u= \left[ \begin{array}{c} \lambda t + \langle x,v\rangle \\ tx \end{array} \right], \ \ \ H^2_u= \left[ \begin{array}{c} \langle y,v\rangle \\ ty + A_v \end{array} \right], \ \ \ H^3_u= \left[ \begin{array}{c}  \langle z,v\rangle \\ tz + B_v \end{array} \right]
\ee
where $\langle x,v\rangle$ means the usual dot product. Their wedge product is zero by the width-$k$ condition. The components containing $e_1$ give us
\be \label{eq:threewedge}
(\lambda t+\langle x,v\rangle ) (ty+A_v)\wedge (tz+B_v) - \langle y,v\rangle tx \wedge (tz+B_v) + \langle z,v\rangle tx \wedge (ty+A_v) = 0.
\ee
This is a cubic polynomial in $t$.
By looking at the $t^3$ term, we get
$$ \lambda y\wedge z=0. $$
So \emph{$y$ and $z$ are parallel}. If both of them are zero, then (\ref{eq:threewedge}) becomes
$$ (\lambda t+ \langle x, v\rangle ) A_v \wedge B_v =0.$$
This implies that $A_v\wedge B_v=0$ for any $v\in V'$. By Lemma \ref{codim1}, we know that $A$ is proportional to $B$, so $\{ H^2, H^3\}$ is linearly dependent (recall that we have assumed that $a=b=0$!), a contradiction. Thus $y$ and $z$ cannot be both zero. Without loss of generality, let us \emph{assume that $y\neq 0$}. Then $z=c\,y$ for some constant $c$. Replace $H^3$ by $H^3-cH^2$, we may \emph{assume that $z=0$}. The equation (\ref{eq:threewedge}) now takes the form
 \be \label{eq:1}
(\lambda t+\langle x,v\rangle ) (ty+A_v)\wedge B_v - \langle y,v\rangle tx \wedge B_v = 0.
\ee
Letting $t=0$, we get
$$  \langle x,v\rangle A_v\wedge B_v=0$$
for any $v\in V'$. If $x\neq 0$, then for generic $v\in V'$, $\langle x,v\rangle \neq 0$, so we know that $A_v\wedge B_v=0$ for generic hence all $v\in V'$. On the other hand, if $x=0$, then (\ref{eq:1}) becomes
$$ t^2 \lambda  y\wedge B_v + t \lambda A_v\wedge B_v =0,$$
so again we have $A_v\wedge B_v=0$. Thus $A$ and $B$ are proportional by Lemma \ref{codim1}. \emph{We have $B\neq 0$} since we have assumed that $b=0$ and $z=0$, and so $A=cB$ for some constant $c$. Replacing $H^2$ by $H^2-cH^3$, we may \emph{assume that $A=0$}. Now (\ref{eq:1}) simply means
$$ y\wedge B_v=0, \ \ \ \langle y,v\rangle x\wedge B_v=0 $$
for all $v\in V'$. Since $y\neq 0$ and $B\neq 0$, as otherwise $H^2$ or $H^3$ would be zero, the first equation above implies that $B$ is a constant multiple of the rank one matrix $y\,y^t$, while the second equation implies that $x=c\,y$ for some constant $c$. With all these assumptions (\ref{threematrix}) now reduces to
\be\label{newthreematrix}
H^1= \left[ \begin{array}{ll} \lambda & (cy)^t \\ cy & 0 \end{array} \right], \ \ \ \ H^2= \left[ \begin{array}{ll} 0 & y^t \\ y & 0 \end{array} \right], \ \ \ \ H^3= \left[ \begin{array}{ll} 0 & 0 \\ 0 & y\,y^t \end{array} \right].
\ee
Now take any $0\neq v_0\in V'$ with $\langle v_0,y\rangle =0$, we know that $v_0$ lies in the common kernel of all three $H^i$ in (\ref{newthreematrix}), a contradiction. This completes the argument for this subcase.

{\bf Subcase 1b: $\lambda = 0$, $(a,b)\neq (0,0)$.}

Without loss of generality, let us assume that $b=1$. Replace $H^2$ by $H^2-aH^3$, we may \emph{assume that $a=0$}. The three column vectors (\ref{threecolumn}) now become
\be\nonumber
H^1_u= \left[ \begin{array}{c}  \langle x,v\rangle \\ tx \end{array} \right], \ \ \ H^2_u= \left[ \begin{array}{c} \langle y,v\rangle \\ ty + A_v \end{array} \right], \ \ \ H^3_u= \left[ \begin{array}{c}  t+\langle z,v\rangle \\ tz + B_v \end{array} \right] ,
\ee
and the  components containing $e_1$ in their wedge product give us
\be \label{eq:threewedge2}
\langle x,v\rangle  (ty+A_v)\wedge (tz+B_v) - \langle y,v\rangle tx \wedge (tz+B_v) + (t+\langle z,v\rangle )tx \wedge (ty+A_v) = 0.
\ee
By looking at the highest and lowest order terms in $t$, we get
$$ x\wedge y=0, \ \ \ \ \ \ \langle x,v\rangle A_v\wedge B_v=0.$$
Note that $x\neq 0$ since $H^1\neq 0$, so the above equations lead to $y=c\,x$ for some constant $c$, and $A_v\wedge B_v=0$ for any generic hence for all $v\in V'$. Thus $A$ and $B$ are proportional to each other, still by Lemma \ref{codim1}. Replacing $H^2$ by $H^2-cH^1$, we may \emph{assume that $y=0$}. This implies that $A\neq 0$. So $B=c'A$ for some constant $c'$. Replacing $H^3$ by $H^3-c'H^2$, we may \emph{assume that $B=0$}. Now (\ref{eq:threewedge2}) gives us
$$ \langle x,v\rangle \,z\wedge A_v=0, \ \ \ \ \ x\wedge A_v=0$$
for any $v\in V'$. This means $A=c \,x\,x^t$ and $z=c'x$ for some constants $c$, $c'$. Once again it will lead to the non-triviality of the common kernel of the three matrices $H^i$.

{\bf Subcase 1c: $\lambda = a=b=0$.}

In this case the three column vectors become
$$
H^1_u= \left[ \begin{array}{c}  \langle x,v\rangle \\ tx \end{array} \right], \ \ \ H^2_u= \left[ \begin{array}{c} \langle y,v\rangle \\ ty + A_v \end{array} \right], \ \ \ H^3_u= \left[ \begin{array}{c}  \langle z,v\rangle \\ tz + B_v \end{array} \right] ,
$$
and the  components containing $e_1$ in their wedge product give us
\be \label{eq:2}
\langle x,v\rangle  (ty+A_v)\wedge (tz+B_v) - \langle y,v\rangle tx \wedge (tz+B_v) + \langle z,v\rangle tx \wedge (ty+A_v) = 0.
\ee
Let $t=0$, we get $\langle x, v\rangle A_v\wedge B_v=0$. Since $x\neq 0$ otherwise $H^1=0$, for generic hence all $v$, $A_v\wedge B_v=0$. Thus $A$ and $B$ are proportional. By scrambling  $\{H^2, H^3\}$, we may \emph{assume that $A=0$}. This will force \emph{$x$ and $y$ to be linearly independent} because $\{H^1,H^2\}$ is so. Since the dimension of $V$ is $n\geq4$ by our assumption, the vanishing in the $V'$ part of the wedge product of the three column vectors gives us
\be \label{eq:3}
 tx \wedge ty \wedge (tz+B_v) = 0.
 \ee
So $x\wedge y\wedge z=0$. Hence $z=c \,x + c' y$ for some constants $c$ and  $c'$ as $x$ and $y$ are linearly independent. Replacing $H^3$ by $H^3-cH^1-c'H^2$, we may \emph{assume that $z=0$}. By (\ref{eq:3}), we also have $x\wedge y\wedge B_v=0$ for all $v$. Thus $B$ must live in the plane spanned by $x$ and $y$, that is,
$$ B = \alpha \,x\,x^t + \beta (x\,y^t  + y\,x^t) + \gamma \,y\,y^t   $$
for some constants $\alpha$, $\beta$, $\gamma$. Equation (\ref{eq:2}) now takes the form
$$ (\langle x,v\rangle  y- \langle y,v\rangle x )\wedge B_v = 0.
$$
Combining the above two lines, we get $Q(v) x\wedge y =0$, where
$$ Q(v) = \alpha \langle x,v\rangle^2 + 2\beta \langle x,v\rangle \langle y,v\rangle + \gamma \langle y,v\rangle ^2.$$
Since $B\neq 0$ as otherwise $H^3=0$, the three constants $\alpha$, $\beta$, $\gamma$ cannot be all zero, thus $Q$ cannot be identically zero, a contradiction. This completes the proof of the case when $n>3$.

{\bf Case 2: $n=3$.}

Without loss of generality, let us assume that $H^1$ has rank $p$, and $p$ is the \emph{smallest} for all $H^{w}$. Clearly $p$ is either $1$ or $2$ here. First let us assume that $p=1$. Write the three matrices as
$$
H^1= \left[ \begin{array}{ll} 1 & 0 \\ 0 & 0 \end{array} \right], \ \ \ \ H^2= \left[ \begin{array}{ll} 0 & x^t \\ x & A \end{array} \right], \ \ \ \ H^3= \left[ \begin{array}{ll} 0 & y^t \\ y & B \end{array} \right],
$$
where $x$, $y$ are column vectors in ${\mathbb C}^2$ and $A$, $B$ are $2\times 2$ matrices. For $i=2$ and $3$, we have replaced $H^i$ by $H^i-(H^i_{11})H^1$ to ensure that their $(1,1)$-th entries are both zero. For column vector $u \in {\mathbb C}^3$ in the form $u = \binom{t}{v}$, we have
$$
H^1_u= \left[ \begin{array}{c}  t  \\ 0 \end{array} \right], \ \ \ H^2_u= \left[ \begin{array}{c} \langle x,v\rangle \\ tx + A_v \end{array} \right], \ \ \ H^3_u= \left[ \begin{array}{c}  \langle y,v\rangle \\ ty + B_v \end{array} \right]
$$
The width-$3$ condition says that $(tx+A_v)\wedge (ty+B_v)=0$ for any $t\in {\mathbb C}$ and any $v\in {\mathbb C}^2$. Thus
\be \label{eq:4}
x\wedge y=0, \ \ \ A_v\wedge B_v=0, \ \ \ x\wedge B_v=y\wedge A_v.
\ee
By Lemma \ref{codim1}, $A$ and $B$ are proportional to each other. So after scrambling $\{ H^2, H^3\}$ we may \emph{assume that $A=0$}. Thus $x\neq 0$, hence $y=c\,x$ for some constant $c$. So replace $H^3$ by $H^3-cH^2$, we may \emph{assume $y=0$}. Hence $B\neq 0$.

Now by the thrid equation of (\ref{eq:4}) we get $x\wedge B_v=0$ for any $v$, so $B = c \, x\,x^t$ for some constant $c$. Taking $0\neq v_0\in {\mathbb C}^2$ with $\langle x, v_0\rangle =0$, we know that $v_0$ lies in the common kernel of all three $H^i$, a contradiction.

Next let us assume that $p=2$. Now we may take a basis of $V$ so that the three matrices are in the form
$$
H^1= \left[ \begin{array}{ll} 0 & 0 \\ 0 & I_2 \end{array} \right], \ \ \ \ H^2= \left[ \begin{array}{ll} 0 & ^t\!x \\ x & A \end{array} \right], \ \ \ \ H^3= \left[ \begin{array}{ll} b & ^t\!y \\ y & B \end{array} \right].
$$
Here we scrambled $\{H^2, H^3\}$ to ensure that $H^2_{11}=0$. For $u\in {\mathbb C}^3$ in the form $u =\binom{t}{v}$, we have
$$
H^1_u= \left[ \begin{array}{c} 0  \\ v \end{array} \right], \ \ \ H^2_u= \left[ \begin{array}{c} \langle x,v\rangle \\ tx + A_v \end{array} \right], \ \ \ H^3_u= \left[ \begin{array}{c}  bt+\langle y,v\rangle \\ ty + B_v \end{array} \right] .
$$
We have
$$
 \big[\langle x,v\rangle (ty+B_v)- (bt+ \langle y,v\rangle )\, (tx+A_v) \big] \wedge v = 0.
$$
Hence
\be \label{eq:5}
b x\wedge v=0,\ \ \big(\langle x,v\rangle B_v-\langle y,v\rangle A_v\big)\wedge v =0, \ \ \big(\langle x,v\rangle y-\langle y,v\rangle x + bA_v\big)\wedge v =0.
 \ee
If $b\neq 0$, then $x\wedge v=0$ for all $v\in {\mathbb C}^2$ will force $x=0$. In this case, the only non-trivial part of the matrix $H^2-cH^1$ is the $2\times 2$ block $A-cI_2$. When $c$ equals to an eigenvalue of $A$, this matrix has rank $1$, contradicting with our assumption of the minimal rank $p=2$.  So we must have $b=0$. The same argument shows that $x$, $y$ cannot be zero, and they must be linearly independent for the same reason, as otherwise some $H^w$ will have rank $1$.

Let $V^0\subset {\mathbb C}^2$ be the open dense subset consisting of $v$ such that $\langle x,v\rangle $ and $\langle y,v\rangle $ are not both zero. From the third equation of (\ref{eq:5}), we know that for any $v\in V^0$ there will be a unique constant $c(v)$ such that
$$ v = c(v) \big( \langle x,v\rangle y-\langle y,v\rangle x\big). $$
Taking the dot product with $v$,
we get that $\langle v, v\rangle =0$ for all $v\in V^0$, which is absurd. This completes the proof of Case 2 and Lemma \ref{codim2}.
\end{proof}

\section{Proof of Lemma \ref{codim3}}\label{proof of codim3}

In this section, we will prove Lemma \ref{codim3}.  Throughout it, we will assume that $n\geq 4$ and  ${\mathcal H}=\{ H^1, \ldots , H^r\}$ is a set of symmetric $n\times n$ matrices, satisfying the non-degeneracy condition  and the width-$4$ condition. We will call such a ${\mathcal H}$ simply as a {\em system}.

As before, we will write $H^w=\sum_{i=1}^r a_i H^i$ for $w=(a_1, \ldots , a_r)\in W\cong{\mathbb C}^r$, and replace ${\mathcal H}$ by another more convenient basis from time to time, and will call this a {\em scramble} of ${\mathcal H}$.

\begin{definition}
The system ${\mathcal H}$ is said to be {\em special}, if there exists a linearly independent set $\{ w_1, w_2, w_3\}$ in $W$ such that $\bigcap_{i=1}^3 \ker (H^{w_i})$ is $(n-2)$-dimensional.
\end{definition}

In these terminologies, Lemma \ref{codim3} simply says that, given a system ${\mathcal H}$, then $l:=\text{rank}({\mathcal H})\leq 4$, and if $l=4$ and $n\geq 5$, then the system ${\mathcal H}$ is special. Equivalently, we can rephrase this as, (a) any special system ${\mathcal H}$ has $l\leq 4$;  (b) any non-special system must have $l\leq 4$, and $l=4$ only when $n=4$.

\begin{lemma}
Let ${\mathcal H}$ be a special system, then $l\leq 4$.
\end{lemma}

\begin{proof}
By a scramble if necessary, we may assume that $H^1,\ldots,H^l$ are linear independent and given by
$$
H^1=\left[ \begin{array}{ccc} 1 & 0 & 0 \\ 0 & 0 & 0 \\ 0 & 0 & 0  \end{array} \right], \ \ \
H^2=\left[ \begin{array}{ccc} 0 & 0 & 0 \\ 0 & 1 & 0 \\ 0 & 0 & 0  \end{array} \right], \ \ \
H^3=\left[ \begin{array}{ccc} 0 & 1 & 0 \\ 1 & 0 & 0 \\ 0 & 0 & 0  \end{array} \right], \ \ \
H^i=\left[ \begin{array}{lll} 0 & 0 & x_i^t \\ 0 & 0 & y_i^t \\ x_i & y_i & A^i  \end{array} \right]
$$
with $4\leq i \leq l$,  where each $x_i$, $y_i$ is a column vector in ${\mathbb C}^{n-2}$, and each $A^i$ is an $(n-2)\times (n-2)$ symmetric matrix. The upper left $2\times 2$ corner of $H^i$ is zero because we can scramble by adding suitable combination of the first three matrices.

Assume on the contrary that $l\geq 5$. Then we may consider the set $\{ H^1, H^2, H^4, H^5\}$. To avoid too much subscripts, let us write $x_4=x$, $y_4=y$, $A^4 =A$, while $x_5=x'$, $y_5=y'$, $A^5=B$. For column vector $u$ in the form $u^t = (t,s, v^t)$, the vectors $H^i_u$ for $i\in \{ 1,2,4,5\}$ are
$$
\left[ \begin{array}{c} t \\ 0 \\ 0  \end{array} \right], \ \ \ \left[ \begin{array}{c} 0 \\ s \\ 0  \end{array} \right], \ \ \ \left[ \begin{array}{c} \langle x,v\rangle  \\ \langle y,v\rangle  \\ tx+sy+A_v  \end{array} \right], \ \ \
\left[ \begin{array}{c} \langle x',v\rangle  \\ \langle y',v\rangle  \\ tx'+sy'+B_v  \end{array} \right] .
$$
The vanishing of their wedge product leads to
$$
ts\,(tx+sy+A_v) \wedge (tx'+sy'+B_v) =0
$$
for any $t,s\in {\mathbb C}$ and any $v\in {\mathbb C}^{n-2}$. Thus
\begin{eqnarray}\label{ccc}
\left\{\begin{array}{ll}
x\wedge x' \ = \ y \wedge y' \ = \ x\wedge y' + y \wedge x' \ = \ 0 ,\\
x\wedge B_v - x'\wedge A_v \ = \ y\wedge B_v - y'\wedge A_v\  = \ 0,\\
A_v\wedge B_v \  = \ 0.
\end{array} \right.
\end{eqnarray}

The last equation in (\ref{ccc}) implies that $A$ and $B$ are proportional, so by a scramble we may \emph{assume that $A=0$}. Thus $x$ and $y$ cannot be both zero. Without loss of generality, let us \emph{assume that $x\neq 0$}.  We will fix this $H^4$ now.

By $x\wedge x'=0$, we know that $x'=c\,x$ for some constant $c$. Replace $H^5$ by $H^5-cH^4$, we may \emph{assume that $x'=0$}. Now the rest of the equations in (\ref{ccc})  take the form
$$ y\wedge y' = x \wedge y' = 0, \ \ \ x\wedge B_v = y\wedge B_v =0. $$
Therefore $y'=\lambda x$ and $B = \mu \,x\,^t\!x$  for some constants $\lambda$ and $\mu$. Since $x'=0$, we know that $y'$ and $B$ cannot be both zero, and the non-vanishing of either of them will lead to $y\wedge x=0$.

Since $n\geq 4$, we may take  $0\neq v_0\in {\mathbb C}^{n-2}$ so that $\langle x, v_0\rangle =0$. This $v_0$ will lie in the kernel of $H^4$ as well as the kernel of $H^5$. Clearly it also lies in the kernel of $H^i$ for any $i>5$. This violates the non-degeneracy condition, so $l\geq 5$ is impossible.
\end{proof}

Now let us focus on the non-special systems. Again to streamline writings, let us first consider a simpler case, where there is some $H^w$ with rank one.

\begin{lemma}
Let ${\mathcal H}$ be a non-special system that contains a rank one matrix, then $l\leq 4$ and $l=4$ only when $n=4$.
\end{lemma}

\begin{proof}
Assume on the contrary that $l\geq5$.
Without loss of generality, let us assume our $l$ matrices are given as
$$
H^1=\left[ \begin{array}{cc} 1 &  \\  & 0_{n-1}  \end{array} \right], \ \ \
H^i=\left[ \begin{array}{cc} 0  & ^t\!x_i \\ x_i & A^i  \end{array} \right], \ \ \ 2\leq i\leq l,
$$
where each $x_i$ is a column vector in ${\mathbb C}^{n-1}$ and each $A^i$ a symmetric $(n-1)\times (n-1)$ matrix. Still with the column vector $u=\binom{t}{v}$ and $v\in\mathbb{C}^{n-1}$, the column vectors $H^i_u$ now take the form
 $$
H^1_u= \left[ \begin{array}{c} t \\ 0   \end{array} \right], \ \ \ H^i_u=
\left[ \begin{array}{c} \langle x_i,v\rangle    \\ tx_i+A^i_v  \end{array} \right] .
$$
We will take four matrices, $\{H^1, H^i, H^j, H^k\}$, where $2\leq i<j<k\leq l$.  The width-$4$ conditions now gives
\be \label{rankone}
(tx_i+A^i_v) \wedge (tx_j+A^j_v)  \wedge (tx_k+A^k_v)  = 0,
\ee
In particular, $x_i\wedge x_j\wedge x_k=0$, so the space $V_x$ spanned by $\{ x_2, \ldots , x_l\}$ has dimension $p\leq 2$. Note that the case $p=0$ cannot occur, as in this case all $x_i=0$, thus $\{ A^2, \ldots , A^l\}$ is linearly independent, satisfies the non-degeneracy condition, and also satisfies the width-$3$ condition as by (\ref{rankone}) we have $A^i_v\wedge A^j_v\wedge A^k_v=0$ for any $v$. This will make the system ${\mathcal H}$ to be special due to Lemma \ref{codim2}. We are left with two possibilities: $p=1$ or $p=2$.

{\em Case 1: $p=1$.}

We show that in this case $l\leq3$.

By a scramble we may assume that $x_2\neq 0$, and $x_3=\cdots = x_l=0$. If $l\geq 5$, then by  $A^3_v\wedge A^4_v \wedge A^5_v=0$ we know that ${\mathcal H}$ is special due to Lemma \ref{codim2}, a contradiction. So we must have $l\leq4$.

If $l=4$. Then $\{ A^3, A^4\}$ is linearly independent, and the equation (\ref{rankone}) becomes
\be \label{rankone2}
A^2_v\wedge A^3_v\wedge A^4_v=0, \ \ \ \ x_2\wedge A^3_v\wedge A^4_v=0.
\ee
If $\{ A^2, A^3, A^4\}$ is linearly independent, then by the first equation in (\ref{rankone2}) and Lemma \ref{codim2} we know that $\mathcal{H}$ is special, a contradiction. So $\{ A^2, A^3, A^4\}$ is linearly dependent. The second equation in (\ref{rankone2}) implies that $x_2$ belongs to that plane. So $A^2$ must be a linear combination of $A^3$ and $A^4$. By a scramble, we may \emph{assume that $A^2=0$}. The second equation in (\ref{rankone2}) implies that $(x_2\,x_2^t)_v\wedge A^3_v\wedge A^4_v=0$. So again $\{x_2\,x_2^t, A^3, A^4\}$ is linearly dependent, meaning that a linear combination of $A^3$ and $A^4$ is equal to $x_2\,x_2^t$. By a scramble, we may \emph{assume that $A^3=x_2\,x_2^t$}. Then $\{H^1, H^2, H^3\}$ makes ${\mathcal H}$ special, a contradiction to our assumption. So $l$ cannot be $4$.

{\em Case 2: $p=2$.}

We may assume that $x_2\wedge x_3\neq 0$ and $x_4= \cdots =x_l=0$. By (\ref{rankone}), we get $x_2\wedge x_3\wedge A^4_v=0$ for any $v$. This means that
$$ A^4=a \,x_2\,x_2^t +b \,(x_2\,x_3^t + x_3\,x_2^t) + c \,x_3\,x_3^t$$
for some constants $a$, $b$, $c$. For simplicity, we will denote this by $R(A^4)\subset \mbox{sp} \{x_2, x_3\}$, and say that the `range' of $A^4$ is contained in the plane spanned by $x_2$ and $x_3$. Note that this can be made precise under appropriate frames, and this loose description will not affect the correctness of the argument. If $l\geq 5$. Then $\{ A^4, A^5\}$ is linearly independent, we have $R(A^4)\subset P$ and $R(A^5)\subset P$ where $P = \mbox{sp} \{x_2, x_3\}$. By considering the wedge product equation for $\{ H^1, H^2, H^4, H^5\}$, we get $A^2_v \wedge A^4_v\wedge A^5_v=0$. If $\{ A^2,  A^4,  A^5\}$ is linearly independent, then they form $2\times 2$ system, hence $R(A^2)\subset P$. If $\{ A^2,  A^4,  A^5\}$ is linearly dependent, then $A^2$ must be a linear combination of $A^4$ and $A^5$ as the latter two are independent. So again we will have $R(A^2)\subset P$. Similarly, $R(A^3)\subset P$, and of course $R(A^i)\subset P$ for $i>5$ if any. This means that  ${\mathcal H}$ forms a $3\times 3$ system, contradicting with the non-degeneracy condition. Hence we must have $l\leq4$.

If $l=4$. By (\ref{rankone}) we also have $A^2_v \wedge A^3_v\wedge A^4_v=0$.  If $\{ A^2, A^3, A^4\}$ is linearly dependent, then after a scramble we may assume that $A^2=0$. Equation (\ref{rankone}) gives us
$$ x_2 \wedge x_3 \wedge A^4_v=0, \ \ \ x_2\wedge A^3_v\wedge A^4_v=0. $$
The first one says that $R(A^4)\subset P:=\mbox{sp}\{x_2, x_3\}$. If $A^4$ is not a multiple of $x_2\,x_2^t$, then the second equation implies that for generic, hence all $v$, we have $x_2\wedge x_3\wedge A^3_v =0$, so $R(A^3)\subset P$, thus ${\mathcal H}$ forms a $3\times 3$ system, a contradiction. On the other hand, if $A^4$ is a constant multiple of $x_2\,x_2^t$, then $\{ H^1, H^2, H^4\}$ form a $2\times 2$ system so ${\mathcal H}$ is special, a contradiction.

Therefore $\{ A^2, A^3, A^4\}$ must be linearly independent. Then the vanishing of their wedge product implies that they form a $2\times 2$ system. If $A^4$ has rank $2$ here, then $R(A^2)$ and $R(A^3)$ are contained in the plane $R(A^4)$ which is $P$, so ${\mathcal H}$ will form a $3\times 3$ system, a contradiction. So $A^4$ must have rank one. By scrambling $H^2$ and $H^3$ if necessary, we may assume that $A^4= x_2\,x_2^t$. In this case, $P$ is spanned by $x_2$ and another vector $x_4\in {\mathbb C}^{n-1}$, where $x_2\wedge x_3 \wedge x_4\neq 0$, and the non-degeneracy condition forces $n$ to be $4$ here. Under the basis $\{ e_1, x_2, x_4, x_3\}$, it is a straightforward computation that $\{ H^1, H^4, H^2, H^3\}$ is in of the following two `normal' forms:
\begin{eqnarray}
&& \Big\{ \left[ \begin{array}{cccc} 1 &  & &  \\ & 0 & &  \\  &  & 0 &  \\  &&& 0 \end{array} \right], \ \ \
\left[ \begin{array}{cccc} 0 &  & &  \\ & 1 & &  \\  &  & 0 &  \\  &&& 0 \end{array} \right], \ \ \
\left[ \begin{array}{cccc} 0 & 1 & 0 &  \\ 1& 0 & 1&  \\ 0 & 1 & 0 &  \\  &&& 0 \end{array} \right], \ \ \
\left[ \begin{array}{cccc}  &  & & 1 \\ &  & 1&  \\  & 1 &  &  \\  1 &&&  \end{array} \right] \Big\}; \\
\nonumber
 \mbox{or}
\\
&&  \Big\{ \left[ \begin{array}{cccc} 1 &  & &  \\ & 0 & &  \\  &  & 0 &  \\  &&& 0 \end{array} \right], \ \ \
\left[ \begin{array}{cccc} 0 &  & &  \\ & 1 & &  \\  &  & 0 &  \\  &&& 0 \end{array} \right], \ \ \
\left[ \begin{array}{cccc} 0 & 1 & 0 &  \\ 1& 0 &0 &  \\ 0 & 0 & 1 &  \\  &&& 0 \end{array} \right], \ \ \
\left[ \begin{array}{cccc}  &  & & 1 \\ &  0 & 0 &  \\  & 0  & 1 &  \\ 1 &&&  \end{array} \right] \Big\}.
\end{eqnarray}
In particular, $l=4$ would imply that $n=4$. This completes the proof of the lemma.
\end{proof}

After these two lemmas, now we may assume that the system ${\mathcal H}$ is not special and does not contain any rank one matrix. We want to show that $l\leq 4$, and $l=4$ only if $n=4$.  We start with the first part.

\begin{lemma} \label{partone}
Let ${\mathcal H}$ be a non-special system without any rank one matrix. Then $l\leq 4$.
\end{lemma}

\begin{proof}
Assume the contrary, namely, $l\geq 5$. Let $n\geq 4$ be the smallest dimension so that such a system exists.

{\em Claim:  Some linear combination in ${\mathcal H}$ has rank $2$.}
\begin{proof}
If $n=4$, the set of all symmetric  $4\times 4$  matrices is $S^2{\mathbb C}^4 \cong {\mathbb C}^{10}$. Denote by $\Sigma$ the subset of matrices with rank at most $2$, then it is easy to see that $\Sigma$ has dimension $7$. So ${\mathbb P}(\Sigma)$ is a $6$-dimensional subvariety in ${\mathbb P}^9$, thus any ${\mathbb P}^3$ in $\mathbb{P}(S^2{\mathbb C}^4)\cong{\mathbb P}^9$ will intersect ${\mathbb P}(\Sigma)$. That is, for any linearly independent set of $4$ or more symmetric $4\times 4$ matrices, some combination will have rank equal to $2$ or less.

If $n\geq 5$, we may take a generic hyperplane $V'$ in $V\cong {\mathbb C}^n$, and restrict the system onto $V'$. By Lemma \ref{hyperplane}, the restriction system will again satisfy the non-degeneracy condition. So by the `minimality' of $n$, the restricted system can no longer be linearly independent, thus a linear combination $H^w$ will have zero restriction on $V'$, which implies  that $H^w$ has rank at most $2$, and the claim is proved.
\end{proof}

By a scrambling if necessary, we may assume that $H^1$ has rank $2$, and the system ${\mathcal H}$ is given by
\be \label{ranktwomatrix}
H^1=\left[ \begin{array}{ccc} 1 &  &  \\  & 1 & \\  &  & 0_{n-2}  \end{array} \right], \ \ \
H^i=\left[ \begin{array}{ccc} 0 & b_i & ^t\!x_i \\ b_i & c_i & ^t\!y_i \\ x_i & y_i & A^i  \end{array} \right], \ \ \ \ 2\leq i\leq l,
\ee
where $x_i$, $y_i$ are column vectors and $A^i$ are symmetric $(n-2)\times (n-2)$ matrices. The $(1,1)$-th position of $H^i$ is zero because we may subtract a suitable multiple of $H^1$ from it.  For column vector $u$ in the form $u = (t,s, v^t)^t$ where the column vector $v\in {\mathbb C}^{n-2}$, we have
\be \label{ranktwocolumns}
H^1_u= \left[ \begin{array}{c} t \\ s \\ 0   \end{array} \right], \ \ \ H^i_u=
\left[ \begin{array}{c} b_is+\langle x_i,v\rangle    \\ b_it+ c_is + \langle y_i,v\rangle \\ tx_i+sy_i+A^i_v  \end{array} \right] , \ \ \ \ \ 2\leq i\leq l.
\ee
For any $2\leq i<j<k\leq l$, the vanishing of wedge product $H^1_u\wedge H^i_u\wedge H^j_u\wedge H^k_u$ first of all gives
$$
(tx_i+sy_i+ A^i_v) \wedge (tx_j+sy_j+ A^j_v)\wedge (tx_k+sy_k+ A^k_v) = 0,
$$
or equivalently
\begin{eqnarray}
&& x_i\wedge x_j\wedge x_k=y_i\wedge y_j\wedge y_k=0,  \nonumber \\
&&  A^i_v\wedge A^j_v \wedge A^k_v=0,  \nonumber \\
&&  {\mathfrak S} \{ x_i \wedge x_j\wedge  y_k\} ={\mathfrak S}\{  x_i \wedge y_j \wedge y_k \} =0,  \label{AAA} \\
&& {\mathfrak S} \{ x_i \wedge x_j\wedge  A^k_v\} = {\mathfrak S} \{ y_i \wedge y_j\wedge  A^k_v\} = {\mathfrak S} \{ (x_i \wedge y_j + y_i \wedge x_j) \wedge  A^k_v\} = 0,  \nonumber \\
&& {\mathfrak S} \{ x_i \wedge A^j_v\wedge  A^k_v\} ={\mathfrak S} \{ y_i \wedge A^j_v\wedge  A^k_v\} = 0, \nonumber
\end{eqnarray}
where ${\mathfrak S}$ means the cyclic sum, namely when $(ijk)$ are cyclicly permuted. By looking at the terms in  $H^1_u\wedge H^i_u \wedge H^j_u \wedge H^k_u$ involving $e_1\wedge e_2$, we get
$$
 {\mathfrak S} \{ \,Q_i(t,s) \,( tx_j + sy_j + A^j_v) \wedge  ( tx_k + sy_k + A^k_v) \} =0.
 $$
 where
 $$
 Q_i(t,s) = b_i(t^2-s^2) + c_i ts + \langle y_i,v\rangle t - \langle x_i,v\rangle s.
 $$
 This is a degree $4$ polynomial in $t$ and $s$, and by looking at the coefficients, we get a bunch of equations. The degree $1$ terms give
 \be
{\mathfrak S} \{  \langle y_i,v\rangle\, A^j_v \wedge A^k_v \} =
 {\mathfrak S} \{  \langle x_i,v\rangle \,A^j_v \wedge A^k_v  \} =0. \label{degree1}
\ee
The degree $2$ terms give
 \begin{eqnarray}
&& {\mathfrak S} \{ b_iA^j_v\wedge A^k_v+  \langle y_i,v\rangle ( x_j\wedge A^k_v - x_k\wedge A^j_v) \} =0, \nonumber \\
&& {\mathfrak S} \{ -b_iA^j_v\wedge A^k_v - \langle x_i,v\rangle ( y_j\wedge A^k_v - y_k\wedge A^j_v) \} =0, \label{degree2} \\
&& {\mathfrak S} \{ c_iA^j_v\wedge A^k_v - \langle x_i,v\rangle ( x_j\wedge A^k_v - x_k\wedge A^j_v) + \langle y_i,v\rangle ( y_j\wedge A^k_v - y_k\wedge A^j_v)   \}  \} =0. \nonumber
\end{eqnarray}
The degree $3$ terms give
 \begin{eqnarray}
&& {\mathfrak S} \{  b_i( x_j A^k_v - x_k A^j_v)   + \langle y_i,v\rangle x_jx_k\} =0, \nonumber \\
&& {\mathfrak S} \{  -b_i  ( y_j A^k_v - y_k A^j_v) - \langle x_i,v\rangle y_jy_k \} =0,  \label{degree3} \\
&& {\mathfrak S} \{  b_i ( y_j A^k_v - y_k A^j_v)  + c_i ( x_j A^k_v - x_k A^j_v) +   \langle y_i,v\rangle ( x_j y_k + y_jx_k) -  \langle x_i,v\rangle  x_jx_k \} = 0, \nonumber \\
&& {\mathfrak S} \{  c_i ( y_j A^k_v - y_k A^j_v) - b_i ( x_j A^k_v - x_k A^j_v)   -  \langle x_i,v\rangle ( x_j y_k + y_jx_k) + \langle y_i,v\rangle  y_jy_k \} = 0, \nonumber
\end{eqnarray}
and finally, the degree $4$ terms give
\begin{eqnarray}
&& {\mathfrak S} \{  b_ix_jx_k \} = {\mathfrak S} \{  b_iy_jy_k \} = {\mathfrak S} \{  c_i ( x_jy_k+y_jx_k) \}  =0, \nonumber \\
&& {\mathfrak S} \{  b_i ( x_jy_k+y_jx_k) + c_i x_jx_k \}   =0, \label{degree4} \\
&& {\mathfrak S} \{  b_i ( x_jy_k+y_jx_k) - c_i y_jy_k \}   =0. \nonumber
\end{eqnarray}

Let $V_x$, $V_y$ be respectively the space spanned by $\{ x_2, \ldots , x_l\}$ or $\{ y_2, \ldots , y_l\}$, and denote by $p_x$, $p_y$ their dimensions. Then we have $p_x, p_y \leq 2$ due to the first equation in (\ref{AAA}). Without loss of generality, let us assume that $p_x\geq p_y$.

{\em Claim: $p_x=2$.}
\begin{proof}
If $p_x=0$, then $p_y=0$, all the $x_i=y_i=0$. There are at least four $A^i$ satisfying the width-$3$ condition by (\ref{AAA}). So they cannot be all linearly independent, otherwise by Lemma \ref{codim2} $\mathcal{H}$ is special and contradicts to the assumption in Lemma \ref{partone}. So we may, by a scrambling if necessary, \emph{assume that $A^2=0$}. In this case $H^2$ has only the upper left $2\times 2$ corner, so $H^2-\lambda H^1$ for suitable $\lambda$ will have rank one, a contradiction. So we must have $p_x\geq 1$.

If $p_x=1$, then we may assume that $x_2\neq 0$ and $x_3= \cdots = x_l=0$. By (\ref{degree1}), we get $\langle x_2, v\rangle A^j_v\wedge A^k_v =0$ for any $v$ and any $2<j<k$. Since $x_2\neq 0$, $\langle x_2, v\rangle \neq 0$ for generic $v$, so for generic $v$ hence all $v$ we have $A^j_v \wedge A^k_v=0$. Thus $A^3, \ldots , A^l$ are proportional to each other. With a scramble, we may assume that $A^4= \cdots = A^l=0$. Now since $p_y\leq 1$, $y_4$ and $y_5$ are proportional, so by a scramble of $\{ H^4, H^5\}$, we may assume that $y_4=0$. Now $H^4$ has only the upper left $2\time 2$ corner, by subtracting a multiple of $H^1$, it will have rank $1$, a contradiction. This competes the proof of the claim.
\end{proof}

So we have $p_x=2$. Let us assume that $x_2\wedge x_3\neq 0$, and $x_4= \cdots = x_l=0$. By the first equation of (\ref{degree4}) applied to the cyclic permutation $(23i)$ with $i\geq 4$, we get $b_ix_2\wedge x_3=0$, hence
\be
b_4= \cdots = b_l=0. \label{b=0}
 \ee
By (\ref{degree1}), we have $\langle x_2, v\rangle A^4_v\wedge A^5_v=0$, hence $A^4_v\wedge A^5_v=0$ for all $v$ so $A^4$, $A^5$ are proportional. Similarly, {\em all $A^4, \ldots , A^l$ are mutually proportional}. Also, by the third equation in (\ref{AAA}) and the fact $x_4=0$, we have $x_2\wedge x_3\wedge y_4=0$, so $y_4 \in P_x=\mbox{sp}\{ x_2, x_3\}$. Similarly, {\em $y_i\in P_x$ for any $i\geq4$}.

If $y_4=y_5=0$, then a linear combination of $H^4$ and $H^5$ will have its lower right corner vanishes, thus with only its upper left $2\times 2$ corner possibly non-zero. By subtracting a multiple of $H^1$ from it, we get a rank $1$ matrix, a contradiction.

If $y_4\wedge y_5\neq 0$, then $P_y=\mbox{sp}\{ y_4, y_5\} = P_x$. By the fourth equation in (\ref{AAA}) and by our assumption that $p_y\leq p_x=2$, $y_4\wedge y_5 \wedge A^i_v=0$ for any $i\neq 4,5$. So $R(A^i)\subseteq P_y$. Similarly, $x_2\wedge x_3\wedge A^j_v=0$ for any $j\neq 2,3$, so $R(A^j)\subseteq P_x$. Therefore, all $x_i$, $y_i$ and all $A^i$ have range in $P=P_x=P_y$, so ${\mathcal H}$ forms a $3\times 3$ system, a contradiction.

We are left with the case when $y_4$ and $y_5$ are not both zero but $y_4\wedge y_5=0$. Without loss of generality, let us assume that $y_4\neq 0$ and $y_5=0$. Since $b_5=0$, $x_5=y_5=0$,  the second line of (\ref{degree4}) applied to $(ijk)=(235)$ gives us $c_5x_2\wedge x_3=0$. Hence $c_5=0$. Also, since $A^4_v\wedge A^5_v=0$, by the first line of (\ref{degree2}) applied to $(245)$, we get $\langle y_4,v\rangle x_2\wedge A^5_v=0$ for all $v$. Hence $x_2\wedge A^5_v=0$ for all $v$, and $A^5$ is proportional to $x_2\,x_2^t$. Now $b_5=c_5=0$, $x_5=y_5=0$, so $H^5$ has rank $1$, a contradiction. This completes the proof of the lemma.
\end{proof}

\begin{lemma} \label{parttwo}
Let ${\mathcal H}$ be a non-special system without any rank one matrix and $l=4$. Then $n=4$.
\end{lemma}

\begin{proof}
Assume that the conclusion fails, namely, there exists a non-special system ${\mathcal H} = \{ H^1, \ldots , H^4\}$ without rank one element such that $n>4$. We want to derive at a contradiction, and the proof will be analogous to that of Lemma \ref{partone}, except that we need a lot more argument since we don't have $H^5$ to help us now.

We may assume that $n\geq 5$ is the \emph{smallest} dimension where such a system exists. If $n>5$, then by restricting the system onto a generic hyperplane of $V={\mathbb C}^n$ and applying Lemma \ref{hyperplane}, we know that there will be some $H^w$ in the system with rank $2$. When $n=5$ this trick can no longer be used, and \emph{we will discuss this case separately in Appendix \ref{proof of codim3-2}}, to rule out the possibility of a system of $4$ symmetric $5\times 5$ matrices where no $H^w$ can be of rank $2$ or lower. So \emph{from now on we will assume that $n\geq 5$ and $H^1$ has rank $2$}.

We assume that the system is given by (\ref{ranktwomatrix}), with column vectors given by (\ref{ranktwocolumns}). The width-$4$ condition gives us equations (\ref{AAA}) through (\ref{degree4}). Let $P_x$, $P_y$ and $p_x$, $p_y$  be as before, and assume that $p_x\geq p_y$. Again we have $p_x\leq 2$.

If $p_x=p_y=0$, then we notice that $\{ A^2, A^3, A^4\}$ must be linearly independent, as otherwise we may assume by a scramble that $A^2=0$ hence $H^2$ will only have its upper left $2\times 2$ corner, and some $H^2-cH^1$ will be rank $1$, a contradiction. By (\ref{AAA}), $A^2_v\wedge A^3_v \wedge A^4_v=0$ for any $v$, thus $\{ A^2, A^3, A^4\}$ forms a $2\times 2$ system, hence ${\mathcal H}$ forms a $4\times 4$ system, a contradiction. So we must have $p_x=1$ or $p_x=2$.

{\em Case 1: $p_x=1$.}

Assume $x_2\neq 0$ and $x_3=  x_4=0$. Since $p_y\leq 1$, we may assume that $y_4=0$ while $y_2\wedge y_3=0$. We will further divide the discussion into two subcases: (a) $y_3\neq 0$, and (b) $y_3=0$.

{\em Subcase 1a: $y_3\neq 0$.}

In this case, by a scramble we may assume that $y_2=0$. Since $x_4=y_4=0$, the matrix $A^4$ cannot be zero, as otherwise some $H^4-cH^1$ will have rank one. By (\ref{degree1}), we get
$$\langle y_3,v\rangle A^2_v\wedge A^4_v=0, \ \ \ \langle x_2,v\rangle A^3_v\wedge A^4_v=0.$$
So $A^2$, $A^3$ are proportional to $A^4$. Subtract multiples of $H^4$ from $H^2$ and $H^3$, we may assume that $A^2=A^3=0$. Now by (\ref{degree2}), we get $ \langle y_3,v\rangle x_2\wedge A^4_v=0$, which implies that $x_2\wedge A_v$ for generic hence all $v$, so $A^4$ is a multiple of $x_2\,^t\!x_2$. Regardless of whether $y_3$ is parallel to $x_2$ or not, the system ${\mathcal H}$ has dimension at most $3$, a contradiction.

{\em Subcase 1b: $y_3=0$.}

In this case we have $x_3=x_4=y_3=y_4=0$, so $\{ A^3, A^4\}$ is linearly independent, as otherwise a linear combination of $H^1$, $H^3$ and $H^4$ would have rank $1$. On the other hand,  by (\ref{degree1}) we have $\langle x_2,v\rangle A^3_v\wedge A^4_v=0$. So for generic hence all $v$, $A^3_v\wedge A^4_v=0$, which will force $\{ A^3, A^4\}$ to be linearly dependent, a contradiction. This completes the proof of Case 1.

{\em Case 2: $p_x=2$.}

Let us assume that $x_2\wedge x_3\neq 0$ and $x_4=0$. By (\ref{AAA}), we get $x_2x_3y_4=0$, so $y_4\in P_x$, and $x_2x_3A^4_v=0$, so $R(A^4)\subset P_x$. On the other hand, by (\ref{degree4}) and (\ref{degree3}), we get $b_4x_2\wedge x_3=0$ so $b_4=0$, and
\be
c_4x_2x_3 +(b_2x_3-b_3x_2)y_4=0, \ \ \ \langle y_4,v\rangle x_2x_3 + (b_2x_3-b_3x_2)A^4_v=0.
\ee
If $y_4=0$, then $c_4=0$, and $(b_2x_3-b_3x_2)A^4_v=0$. Since both $b_4=c_4=0$, we know that $A^4$ must have rank at least two since $H^4$ does, so the vector $b_2x_3-b_3x_2=0$, which implies that $b_2=b_3=0$. The third line of (\ref{degree3}) now gives $(c_2x_3-c_3x_2)A^4_v=0$, so $c_2x_3-c_3x_2=0$, thus $c_2=c_3=0$. Note that all $b_i$ and $c_i=0$, we have $H^2_u\wedge H^3_u \wedge H^4_u=0$, so ${\mathcal H}$ is special, a contradiction.

Therefore we must have $y_4\neq 0$. Let us divide the discussion into three subcases: (a) $p_y=1$, (b) $p_y=2$ and $P_y=P_x$, and (c) $p_2=2$ but $P_y\neq P_x$.

{\em Subcase 2a: $p_y=1$.}

Since $y_4\neq 0$, by a scramble we may assume that $y_2=y_3=0$, and $y_4=x_2$. In this case, by (\ref{degree4}), we get $b_2=c_2=c_4=0$, so only $b_3$, $c_3$ are possibly non-zero. The four equations of (\ref{degree3}) now give us
$$ b_3x_2A^4_v=\langle x_2,v\rangle x_2x_3, \  b_3x_2A^2_v=0, \ \ b_3x_2A^2_v=c_3x_2A^4_v, \ c_3x_2A^2_4+b_3x_2A^4_v+\langle x_2,v\rangle x_2x_3=0. $$
Plug the first two into the last two, we get
$$c_3x_2A^4_v=0, \ \ \ c_3x_2A^2_v = -2 \langle x_2,v\rangle x_2x_3.$$
Since $x_2\neq 0$, the second equation in the above line tells us that $c_3\neq 0$, so the first equation in this line implies that $x_2A^4_v=0$ for all $v$. Plug this into the first equation involving $b_3$, we get $0=\langle x_2,v\rangle x_2x_3$, which means $0=\langle x_2, v\rangle$ for all $v$, a contradiction.

{\em Subcase 2b: $y_4\neq 0$ and $P_y=P_x$.}

By a scramble, we may assume that $y_2=0$. In this case, $\{x_2, x_3\}$ and $\{y_3, y_4\}$ are two basis of $P=P_x=P_y$. By (\ref{AAA}), we get $x_2x_3A^4_v=0$, $y_3y_4A^2_v=0$, so $R(A^4)\subset P$ and $R(A^2)\subset P$. To finish the proof in this case, it suffices to show $R(A^3)\subset P$, as it implies that ${\mathcal H}$ forms a $4\times 4$ system. To show this, let us first assume that $\{ A^2, A^4\}$ is linearly independent. If $\{ A^2, A^3, A^4\}$ is linearly dependent, then $A^3$ must be a linear combination of $A^2$ and $A^4$, thus $R(A^3)\subset P$. On the other hand, if $\{ A^2, A^3, A^4\}$ is linearly independent, then since we have $A^2_vA^3_vA^4_v=0$ by (\ref{AAA}), these three $A^i$ form a $2\times 2$ system, so $R(A^2)+R(A^4)=P$ which contains $R(A^3)$.

Now let us assume that $\{ A^2, A^4\}$ is linearly dependent. By (\ref{degree4}), we have $b_4x_2x_3=0$ and $b_2y_3y_4=0$. So $b_2=b_4=0$. By the first equation of (\ref{degree2}), we have
$$ b_3A^2_vA^4_v - \langle y_3,v\rangle x_2A^4_v + \langle y_4,v\rangle (x_2A^3_v - x_3 A^2_v) = 0. $$
The first term is zero since $\{ A^2, A^4\}$ is linearly dependent. Wedge with $x_3$, we get
$$ \langle y_4,v\rangle x_2x_3A^3_v=0 .$$
So $x_2x_3A^3_v=0$ for generic thus all $v$. This means that $R(A^3)\subset P$ and we are done.

{\em Subcase 2c: $y_4\neq 0$ and $P_y\neq P_x$. }

Again by a scramble we may assume that $y_2=0$. We have $x_2x_3y_4=0$ by (\ref{AAA}), so $y_4\in P_x\cap P_y$. Similarly, $x_2\in P_x\cap P_y$ as well. Scale $H^4$, we may assume that $x_2=y_4=z$. It lies in $P_x\cap P_y$, and $\{ x_3, y_3, z\}$ forms a basis of the space $P=P_x+P_y$.

From (\ref{degree4}), we get $b_2=b_4=c_2=c_4=0$. By (\ref{degree3}), we obtain
\begin{eqnarray*}
&& z\wedge \{ b_3A^4_v - \langle z,v\rangle x_3\}  = z \wedge \{ b_3A^2_v - \langle z,v\rangle y_3\}  = 0, \\
&& z\wedge \{ b_3A^2_v-c_3A^4_v+ \langle z,v\rangle y_3\}  = z\wedge \{ c_3A^2_v+b_3A^4_v+ \langle z,v\rangle x_3\}  = 0. \ \ \ \ \
\end{eqnarray*}
Plug the second one in the first line into the first one on the second line, we get
$$ z\wedge \{ -c_3A^4_v+ 2\langle z,v\rangle y_3\} =0. $$
Taking the wedge product of the last equation with $x_3$, we get
$$ c_3 x_3zA^4_v = 2\langle z,v\rangle x_2zy_3. $$
By (\ref{AAA}), we have $x_2x_3A^4_v=0$, while $x_3zy_3\neq 0$ by our assumption, so we get $\langle z,v\rangle =0$ for all $v$, which is absurd. This completes the proof of the case, thus the lemma.
\end{proof}

\section{$5\times 5$ system without rank $2$ elements}\label{proof of codim3-2}

In this appendix, we will show that for any system ${\mathcal H} =\{ H^1, \ldots , H^4\}$  of symmetric $5\times 5$ matrices, there always exists some $H^w$ with rank $2$ or less, and complete the proof of Lemma \ref{codim3}.

\begin{proof}
Let $r$ be the smallest rank of any $H^w$ in the system. Assume the contrary, namely, $r\geq 3$. We want to derive at a contradiction. Clearly, $r\leq 4$.

{\em Case 1: The minimum rank of $H^w$ is $4$.}

By a scramble, we may assume that
$$ H^1=\left[ \begin{array}{cc} I_4 & \\ & 0 \end{array} \right] , \ \ H^i  =\left[ \begin{array}{cc} A^i & x_i \\ x_i^t & a_i \end{array} \right], \ \ \ 2\leq i\leq 4,
$$
where each $x_i$ is a column vector in ${\mathbb C}^4$ and each $A^i$ a symmetric $4\times 4$ matrix. For column vector $u$ such that $u=(v^t,t)^t$ where $v\in {\mathbb C}^4$ and $t\in {\mathbb C}$, we have
$$ H^1_u = \left[ \begin{array}{c} u \\  0 \end{array} \right] , \ \ H^i_u = \left[ \begin{array}{c} A^i_v + tx_i \\  \langle x_i,v\rangle + ta_i \end{array} \right] , \ \ \ 2\leq i\leq 4.
$$
The terms containing $e_5$ in $H^1_u\wedge H^2_u\wedge  H^3_u\wedge  H^4_u=0$ gives us
$$ v\wedge {\mathfrak S} \{ (\langle x_i,v\rangle + ta_i) ( A^j_v + tx_j) \wedge (A^k_v + tx_k)\} =0,$$
where the cycle $(ijk)$ runs through all cyclic permutations of $(234)$. This is a cubic polynomial in $t$, and the $t^3$ terms give us
\be
 v \wedge {\mathfrak S} \{ a_ix_j\wedge x_k\} = 0.
 \ee
 By a scramble, we may assume that $a_3=a_4=0$. So the last equation takes the form $v\wedge a_2 x_3\wedge x_4=0$ for all all $v$. Hence $a_2 x_3\wedge x_4=0$. If $a_2\neq 0$, then $x_3\wedge x_4=0$. By a scramble of $\{H^3, H^4\}$, we may assume that $x_4=0$. Now $H^4$ has only the upper left block $A^4$, and $H^4-cH^1$ for a suitable constant $c$ would have rank less than $4$, a contradiction. So we may assume that $a_2=0$.

 In this case, notice that in $H^1_u\wedge H^2_u\wedge  H^3_u\wedge  H^4_u=0$, the terms without $e_5$ also gives us $v\wedge x_2\wedge x_3\wedge x_4=0$, which implies that $x_2\wedge x_3\wedge x_4=0$. So by a scramble we may assume that $x_4=0$ once again, which leads to a contradiction as before. So we know that this case does not occur.

 {\em Case 2: The minimum rank of $H^w$ is $3$.}

Assume that
$$ H^1=\left[ \begin{array}{ccc} I_3 & & \\ & 0 & \\ & & 0  \end{array} \right] , \ \ H^i  =\left[ \begin{array}{ccc} A^i & x_i & y_i \\ x_i^t & a_i & b_i \\ y_i^t & b_i & c_i \end{array} \right], \ \ \ 2\leq i\leq 4,
$$
where each $x_i$, $y_i$  is a column vector in ${\mathbb C}^3$ and each $A^i$ a symmetric $3\times 3$ matrix. For column vector $u$ such that $u=(v^t,t, s)^t$ where  $v\in {\mathbb C}^3$ and $t,s \in {\mathbb C}$, we have
$$ H^1_u = \left[ \begin{array}{c} v \\  0 \\0 \end{array} \right] , \ \ H^i_u = \left[ \begin{array}{c} A^i_v + tx_i +sy_i \\  \langle x_i,v\rangle + ta_i +sb_i \\ \langle y_i,v\rangle + tb_i +sc_i \end{array} \right] , \ \ \ 2\leq i\leq 4.
$$
The terms containing $e_4\wedge e_5$ in $H^1_u\wedge \cdots \wedge H^4_u=0$ are
\be \label{cubic}
v\wedge {\mathfrak S} \{ Q_{ij} (tx_k+sy_k+A^k_v) \} =0,
\ee
where the sum is for $(ijk)$ to take all cyclic permutations of $(234)$, and
\begin{eqnarray*}
 Q_{ij} & = & (\langle x_i,v\rangle + ta_i+sb_i)(\langle y_j,v\rangle + tb_j+sc_j)- (\langle x_j,v\rangle + ta_j+sb_j)(\langle y_i,v\rangle + tb_i+sc_i)\\
 & = & t^2(a_ib_j-b_ia_j) + s^2(b_ic_j-c_ib_j) + ts(a_ic_j-c_ia_j) + (\langle x_i,v\rangle \langle y_j,v\rangle - \langle y_i,v\rangle \langle x_j,v\rangle )\\
 & & +\ t(a_i \langle y_j,v\rangle + b_j \langle x_i,v\rangle  - a_j \langle y_i,v\rangle  - b_i \langle x_j,v\rangle   ) \\
 & & + \   s(b_i \langle y_j,v\rangle + c_j \langle x_i,v\rangle  - b_j \langle y_i,v\rangle  - c_i \langle x_j,v\rangle )
 \end{eqnarray*}
 For $2\leq i\leq 4$, let us write
 $$ B^i = \left[ \begin{array}{ll} a_i & b_i \\ b_i & c_i \end{array} \right].$$
 We will divide the discussions into three subcases, depending the behavior of those $B^i$.

{\em Subcase 2a. $\{ B^2, B^3, B^4\}$ is linearly independent.}

With a scramble, we may assume that the matrices $\{ B^2, B^3, B^4\}$ are given by
$$  \left[ \begin{array}{ll} 1 & 0 \\ 0 & 0 \end{array} \right], \
 \ \ \left[ \begin{array}{ll} 0 & 1 \\ 1 & 0 \end{array} \right], \
 \ \  \left[ \begin{array}{ll} 0 & 0 \\ 0 & 1 \end{array} \right], $$
respectively. In other words, we have $a_2=b_3=c_4=0$ and all other $a_i$, $b_i$, $c_i$ are zero. By the cubic terms (in $t$ and $s$) in (\ref{cubic}), we get
$$ y_2=x_4=0, \ \ \ y_3=x_2, \ \ y_4=x_3. $$
Now by the $s^2$ term in (\ref{cubic}), we get
$$ v\wedge \{ A^2_v -\langle x_2,v\rangle y_3 \} =0.$$
Since $y_3=x_2$, the above equation says that $I_v\wedge (A^2-x_2\,x_2^t)_v =0$ for any $v\in {\mathbb C}^3$, thus
$ A^2-x_2\,x_2^t = c\, I_3$ for some constant $c$. This means that
$$ H^2-cH^1 = \left[ \begin{array}{ccc} x_2\,x_2^t & x_2 & 0 \\ x_2^t & 1 & 0 \\ 0 & 0 & 0 \end{array} \right] = \left[ \begin{array}{c} x_2\\ 1 \\0 \end{array} \right]  \cdot [\, x_2^t ,\,1, \,0 ] ,$$
which has rank $1$, a contradiction.

{\em Subcase 2b: $\{ B^2, B^3, B^4\}$ is linearly dependent but not all zero.}

By a scramble, we may assume that $B^4=0$. The cubic terms in (\ref{cubic}) now gives
$$ v\wedge  \gamma x_4 = v\wedge \alpha y_4 = v\wedge  ( \beta x_4+\gamma y_4) = v\wedge  (\alpha x_4 + \beta y_4) =0,$$
where
$$ \alpha = b_2c_3-c_2b_3, \ \ \beta = a_2c_3 - c_2a_3, \ \ \gamma = a_2b_3-b_2a_3. $$
If $x_4=y_4=0$, then $H^4$ has only the upper left corner, thus $H^4-cH^1$ for suitable constant $c$ would have rank less than $3$, a contradiction. If $x_4\neq 0$, then by the above equations, we get successively $\gamma =0$, $\beta=0$, and $\alpha=0$. Similarly, if $y_4\neq 0$, the same thing holds, so we always have $\alpha =\beta =\gamma =0$. This means that $B^2$ and $B^3$ are proportional. So by a scramble of $\{ H^2, H^3\}$, we may assume that $B^3=0$. Under our case assumption, we have $B^2\neq 0$.

By looking at the terms containing one of $e_4$, $e_5$ but not both in $H^1_u \wedge \cdots \wedge H^4_u=0$, the highest order terms in $t$ and $s$ give us
\begin{eqnarray}
 && v\wedge {\mathfrak S} \{ a_ix_jx_k\} = v\wedge {\mathfrak S} \{ b_ix_jx_k\}= v\wedge {\mathfrak S} \{ b_iy_jy_k\} = v\wedge {\mathfrak S} \{ c_iy_jy_k\} =0, \nonumber \\
 && v\wedge {\mathfrak S} \{ a_i (x_jy_k+y_jx_k)+ b_i x_jx_k\} =  v\wedge {\mathfrak S} \{ b_i (x_jy_k+y_jx_k)+ c_i x_jx_k\} =0, \label{abcxy}\\
 && v\wedge {\mathfrak S} \{ a_iy_jy_k + b_i (x_jy_k+y_jx_k)\} =  v\wedge {\mathfrak S} \{ b_i y_jy_k + c_i (x_jy_k+y_jx_k) \} =0. \nonumber
 \end{eqnarray}
Since $B^3=B^4=0$, the above equations imply that, if $x_3\wedge x_4\neq 0$, then $a_2=b_2=c_2=0$, a contradiction. So we must have $x_3\wedge x_4=0$. Similarly, $y_3\wedge y_4=0$. By a scramble of $\{H^3, H^4\}$, let us assume that $x_4=0$. Then $y_4\neq 0$, as otherwise $H^4$ will have only the upper left corner, so $H^4-cH^1$ for suitable $c$ would have rank less than $3$. Therefore, $y_3= c'y_4$ for some constant $c'$. Replace $H^3$ by $H^3-c'H^4$, we may assume that $y_3=0$. This will imply that $x_3\neq 0$. By (\ref{abcxy}), we also get
$$ a_2 x_3y_4 = b_2x_3y_4 =c_2 x_3y_4=0.$$
So by $B^2\neq 0$ we get $x_2\wedge y_4=0$. Scale $H^3$ if necessary, let us assume that $x_3=y_4$, and denote this non-zero vector in ${\mathbb C}^3$ as $v_0$. By the degree two (in $t$ and $s$) terms in (\ref{cubic}), we get
$$ v \wedge a_2v_0 \langle v_0,v\rangle = - v \wedge c_2v_0 \langle v_0,v\rangle = -2 v \wedge b_2v_0 \langle v_0,v\rangle =0$$
for any $v\in {\mathbb C}^3$. This implies that $a_2=b_2=c_2=0$, so $B^2=0$, a contradiction.

{\em Subcase 2c: $B^2=B^3=B^4=0$.}

Again by looking at the terms containing $e_4$ but not $e_5$ in $H^1_u \wedge \cdots \wedge H^4_u=0$, we get
$$ v\wedge {\mathfrak S} \{ \langle x_i,v\rangle x_j\wedge x_k\} =0.$$
If $x_1\wedge x_2\wedge x_3\neq 0$, let us express $v= \sum_{i=1}^3 v_i x_i$ and plug into the above equality, we get
$$ \langle v, v\rangle x_2\wedge x_3\wedge x_4 =0,$$
or equivalently, $\langle v, v\rangle =0$ for all $v$, which is a contradiction. So $\{ x_2, x_3, x_4\}$ must be linearly dependent. Similarly, $\{ y_2, y_3, y_4\}$ is linearly dependent. By a scramble, let us assume that $x_4=0$. Then $y_4\neq 0$, so by another scramble (without changing $H^4$) we may assume that $y_3=0$. Now by the $t$-terms in (\ref{cubic}), we get
$$
\langle y_4, v\rangle v\wedge \{ \langle x_3, v\rangle  x_2 - \langle x_2, v\rangle  x_3\} =0.
$$
Since $y_4\neq 0$, we can drop the factor $ \langle y_4, v\rangle$. Wedge it with $x_3$, we get $v\wedge \langle x_3, v\rangle x_2\wedge x_3=0$. So $x_2x_3=0$. Similarly, by $x_3\neq 0$, we get $y_2\wedge y_4=0$. Now replacing $H^2$ by $H^2-c H^3 -c'H^4$ for suitable constants $c$, $c'$, we may assume that $x_2=y_2=0$. This will mean $H^2$ has only  the upper left corner, so after subtracting a multiple of $H^1$ from it, we will get a matrix with rank less than $3$. This completes the proof.
\end{proof}

\end{document}